\documentclass[12pt,a4paper]{article}

\usepackage{graphicx}

\usepackage[margin=1.1 in]{geometry}
\usepackage{pifont} 

\usepackage{natbib}
\usepackage[dvipsnames]{xcolor}

\usepackage{amsmath,amssymb,amsthm}

\newtheorem{theorem}{Theorem}
\newtheorem{lemma}{Lemma}
\newtheorem{proposition}{Proposition}
\newtheorem{definition}{Definition}

\newtheorem{acknowledgements}{Acknowledgements}
\newtheorem{remark}{Remark}
\usepackage{mathrsfs} 
\usepackage{hyperref}
\hypersetup{
  colorlinks   = true,    
  urlcolor     = Red,    
  linkcolor    = RoyalBlue,    
  citecolor    = OliveGreen      
}
\newcommand\blfootnote[1]{%
  \begingroup
  \renewcommand\thefootnote{}\footnote{#1}%
  \addtocounter{footnote}{-1}%
  \endgroup
}
\begin{document}  
\hfill 
  \begin{center}
    { \Large Existence and stability of periodic solutions of an impulsive differential equation and application to CD8 T-cell differentiation }
        \end{center}  
  \begin{center}
     {Simon Girel$^{1,2}$ and Fabien Crauste$^{1,2}$}
        \end{center}     
       
\begin{description}      
\item[$ ^1$]\textit{ Univ Lyon, Universit\'e Claude Bernard Lyon 1, CNRS UMR 5208, Institut Camille Jordan, 43 blvd. du 11 novembre 1918, F-69622 Villeurbanne cedex, France}
\item[$ ^2$] \textit{Inria, Villeurbanne, France}
\item[\ding{41}] \textsf{girel@math.univ-lyon1.fr};  \textsf{crauste@math.univ-lyon1.fr}
 \end{description}
\blfootnote{This a preprint version of the article published in \textit{Journal of Mathematical Biology (2018; doi \href{http://dx.doi.org/10.1007/s00285-018-1220-3}{10.1007/s00285-018-1220-3}). The accepted article is available online, in a view-only version, at \href{http://rdcu.be/IcN2}{http://rdcu.be/IcN2.}} } 

 \hrulefill 
\vspace{1 cm} 
\begin{abstract}
Unequal partitioning of the molecular content at cell division has been shown to be a source of heterogeneity in a cell population. We propose to model this phenomenon with the help of a scalar, nonlinear impulsive differential equation (IDE). In a first part, we consider a general autonomous IDE with fixed times of impulse and a specific form of impulse function. We establish properties of the solutions of that equation, most of them obtained under the hypothesis that impulses occur periodically. In particular, we show how to investigate the existence of periodic solutions and their stability by studying the flow of an autonomous differential equation. A second part is dedicated to the analysis of the convexity of this flow. Finally, we apply those results to an IDE describing the concentration of the protein Tbet in a CD8 T-cell, where impulses are associated to cell division, to study the effect of molecular partitioning at cell division on the effector/memory cell-fate decision in a CD8 T-cell lineage. We show that the degree of asymmetry in the molecular partitioning can affect the process of cell differentiation and the phenotypical heterogeneity of a cell population.

\textbf{Keywords: } Impulsive differential equation - Flow convexity - Cellular differentiation - Unequal partitioning - Immune response
\end{abstract}
\pagebreak
\section{Introduction}
\label{intro}
Time-dependent processes are often modelled with continuous differential equations
. Actually, a lot of biological processes are impacted by brief events occurring on a lower time scale \citep{kuehn2015multiple}, for example the evolution of cell phenotype subject to gene expression fluctuation, or population dynamics affected by natural disasters. When such events are brief enough comparing to the process of interest, it can be simpler to consider them as instantaneous. The theory of impulsive differential equations (IDE), initiated by \citet{Myshkis}, provides suitable mathematical tools for modelling such processes subject to perturbations. IDE have been used to study many different phenomena such as the effects of vaccination \citep{Wang2015} or of any stress factor on a cell population \citep{Kou2009}, the effects of human activities (hunting, feeding, etc.) on prey-predator dynamics \citep{Liu2003,Liu2012} or single-species systems \citep{Yan2004}, the consequences of seasonal birth pulses on population dynamics \citep{Tang2002} or dengue fever control via introduction of \textit{Wolbachia}-infected mosquitoes  in a non-infected mosquitoes population \citep{Zhang2016}.  An impulsive equation is defined by a differential equation, which characterises the evolution of a system between two impulses, an impulse criterion, which decides when impulses occur, and a set of impulses functions which define the effect of impulses on the system. For a general overview of IDE theory, we refer the reader to \citet{Bainov89,Bainov1993}. 

 In this paper we focus on the particular case of an autonomous differential equation subject to impulses at fixed times, governed by linear impulse functions. That is, a system in the form of
\begin{equation}
{\renewcommand{\arraystretch}{2}
\left \{ 
\begin{array}{r l l}
\dfrac{\mathbf{d}X(t)}{\mathbf{d}t} =& g(X(t)),&\ t\in \mathbb{R}^+\backslash\{ \tau_k, k\in \mathbb{N}^* \}, \\
X(\tau_k^+)=&(1+\alpha_k)X(\tau_k^-),&\ k\in \mathbb{N}^*, \\
X(0)=&X_0,
\end{array}
\right. \\ }\label{ImpulsiveEquation}
\end{equation}

  We aim at describing the phenomenon of unequal repartition of proteins between daughter cells at cell division \citep{Bocharov2013}, and its consequences on the emergence of different possible fates for a cytotoxic T lymphocyte, known as CD8 T-cell.

  Following infection of the organism by a pathogen, CD8 T-cells are activated by antigen presenting cells (APC), proliferate and develop cytotoxic functions, known as effector functions, to fight the infection. In the meantime, 5 to 10$\%$ of those lymphocytes develop a memory profile characterised by higher survival properties and abilities to react faster to a subsequent infection \citep{Wherry2004}. Once the infection is cleared, effector cells die progressively during the so-called contraction phase while memory cells survive in the organism on a long time scale. Even though the mechanisms controlling the fate of each cell are still not well known, it has been shown that high and increasing levels of protein Tbet \--- a transcription factor expressed by CD8 T-cells and involved in developmental processes \--- in a CD8 T-cell promote the development of effector profile and repress differentiation toward memory phenotype \citep{Joshi2007,Kaech2012,Lazarevic2013}.  
  
 Moreover, \citet{Chang2011,Chang2007} have shown that the activation of a cytotoxic T lymphocyte by an APC induces the polarisation of the lymphocyte so that Tbet, and other key factors involved in cell fate, mainly gather to one side of the plane of division of the cell, giving birth to two daughter cells that inherit different amounts of those determinants. They suggested that the asymmetric first division following activation of the cell results in differently fated daughter cells toward effector or memory lineages. 
Since the polarization of the dividing cell requires APC binding, only the first division after activation can be asymmetric, in the sense that the two daughter cells exhibit clearly distinct fates. However, it has been shown \citep{Block1990,Sennerstam88} that, for subsequent divisions, uneven stochastic partitioning of the cellular content at division is still observed. Once repeated over several divisions, this phenomenon of unequal repartition of the proteins can lead to a strong heterogeneity in a cell population coming from the same initial mother cell, resulting in different phenotypes at the end of the differentiation process.


In 2013, \citet{Bocharov2013} highlighted that mathematical tools should be developed in order to take into account the unequal repartition of proteins at cytotoxic T lymphocyte division, in particular for the carboxyfluorescein succinimidyl ester (CFSE) dye, which is  used to analyse cell proliferation under the hypothesis that it is symmetrically halved between the daughter cells upon cell division.

 Considering a CFSE-labelled lymphocyte population,  \citet{Luzyanina2013} built a system of delay hyperbolic partial differential equations structured by a continuous variable representing intracellular CFSE amount   and allowing uneven distribution of CFSE between daughter cells. Each equation models the size of the population of lymphocytes that have undergone a given number of divisions. They showed that data are better explained by their model when unequal repartition of CFSE between daughter cells at division is taken into account. 

\citet{Mantzaris2006,Mantzaris2007} introduced a variable number Monte Carlo algorithm in which stochastic division effects such as cell cycle time and repartition of proteins to daughter cells are considered. This algorithm can simulate discrete cell population dynamics, starting from a single cell, along with a deterministic description of the evolution of the quantity of an arbitrary protein in each cell, through an ordinary differential equation. Mantzaris also presented a deterministic partial differential equation of the population density, structured by the quantity of intracellular proteins. He compared the results from both the stochastic and deterministic models and showed that they are very close for big enough population sizes, while stochastic effects are more significant in small cell populations.

\citet{Prokopiou} and  \citet{Gao2016} developed a multiscale agent-based model describing a discrete population of CD8 T-cells in a lymph node, in the context of the immune response. A system of differential equations, embedded in each cell, describes the concentrations of six intracellular proteins, including Tbet, which control cell differentiation, death and cytotoxicity. When a cell divides, its molecular content is stochastically partitioned between the two daughter cells, resulting in a heterogeneous population. This model is able to qualitatively and quantitatively reproduce the immune response in a lymph node from the activation of an initial population of CD8 T-cells to the development of their effector functions and the beginning of the clonal expansion phase.

It must be noted that the continuous structured population density approach used by \citet{Luzyanina2013} and \citet{Mantzaris2006,Mantzaris2007} cannot be used to study single cell fate decision, while no formal analysis can be performed on the stochastic computational algorithms presented by \citet{Mantzaris2006,Mantzaris2007}, \citet{Prokopiou} and \citet{Gao2016}. 

We propose a different approach, that does not focus on a population of cells but rather on the concentration of protein Tbet in a single CD8 T-cell subject to multiple divisions. From the equation on protein Tbet used in \citet{Prokopiou} and \citet{Gao2016}, we propose an impulsive equation where the differential equation dynamics describes the regulation of Tbet concentration in the cell while the impulses account for the effect of protein partitioning at division in that cell. To our knowledge, this is the first work dealing with IDE from this point of view, and the first one applied to fate decision making in CD8 T-cells. This approach allows us to use theoretical results about IDEs to investigate effects of protein partitioning on cell fate decision making.

The paper is organized as follows. In Section \ref{PropGénérales} we introduce some properties on the existence, monotonicity and asymptotic behaviour of the solutions of the autonomous IDE with fixed impulse times (\ref{ImpulsiveEquation}), with the main results shown for the particular case where impulses occur periodically and $\alpha_k\equiv \alpha$ for all $k\in \mathbb{N}^*$.  
In Section \ref{SectionSolPeriodiques}, we discuss the existence and stability of periodic solutions of (\ref{ImpulsiveEquation}). The results mostly rely on the properties of the flow associated to (\ref{ImpulsiveEquation}).
In Section \ref{sectionConvex} we study the convexity of that flow. In a first part, we show some preliminary results under the assumption that $g$ is a piecewise linear function and then extend our study to any continuously differentiable function.
In Section \ref{ChapitreImpulseTbet}, the dynamics of the concentration of protein Tbet in a single CD8 T-cell undergoing several divisions is modelled  with an IDE, where an impulse occurs at each division. We use the results from previous sections to study the number of periodic solutions and their stability. 
In Section \ref{SectionBio} we propose an explanation to how a single mother cell can, through multiple divisions, give birth to an heterogeneous cell population, composed of two pools of lymphocytes with opposite phenotypes. 

\section{Impulsive differential equations: definitions and basic properties} \label{PropGénérales}

We consider the impulsive system (\ref{ImpulsiveEquation}), where impulses occur at fixed times $\tau_k$, $k\in \mathbb{N}^*$. Parameters $(\alpha_k)_{k\geq 1}$ and  $(\tau_k)_{k\geq 1}$ are two sequences of real numbers, $g:U\to \mathbb{R}$ is a lipschitz-continuous function, $X_0\in U$ and $U$ is such that either $U=\mathbb{R}$, or $g(0)=0$ and $U\in\{\mathbb{R}^-,\mathbb{R}^+\}$.

The definition of $U$ is such that for all $x\in U$ and for all $\alpha>-1$, we have $(1+\alpha)x\in U$ and, in the case where $U \in\{\mathbb{R}^-,\mathbb{R}^+\}$, the condition $g(0)=0$ ensures that, for any initial condition $X_0\in U$, the solution of the autonomous equation $\mathbf{d} X/\mathbf{d}t=g(X)$ remains in $U$.

\begin{definition}\emph{\citet{Yan1998}}
For any $X_0 \in U$, a real function $X$ defined on $\mathbb{R}^+$ is said to be a solution of (\ref{ImpulsiveEquation}) if the following conditions are satisfied:
\begin{description}
\item[(i)] X is absolutely continuous on $ [0, \tau_1 )$ and on each interval $(\tau_k, \tau_{k+1})$, $k\in \mathbb{N}^* $,
\item[(ii)] for any $\tau_k,\ k \in \mathbb{N}^*,\ X(\tau_k^+)$ and $X(\tau_k^-)$ exist in $\mathbb{R}$ (i.e. $X$ may have discontinuity of the first kind only) and $X(\tau_k^+) = X(\tau_k)$ (i.e. $X$ is right continuous),
\item[(iii)]$X$ satisfies (\ref{ImpulsiveEquation}).
\end{description}
\label{defSolImp} 
\end{definition}
 In the following, we denote the solution of (\ref{ImpulsiveEquation}) at time $t$ by $X(t;X_0,\alpha_k)$. If for a given $\alpha \in \mathbb{R}$ and for all $k \in \mathbb{N}^*$, $\alpha_k=\alpha$,  we simply write $X(t;X_0,\alpha)$. In particular the solution of System (\ref{ImpulsiveEquation}) without impulse is $X(t;X_0,0)$. 
 
We introduce the following hypotheses:\begin{description}
\item[($\mathbf{H_1}$)] $0< \tau_1<\tau_2<\tau_3<\dots$ are fixed values and $\underset{k \to \infty}{lim}\tau_k=+\infty$. 
\item[($\mathbf{H_2}$)] $(\alpha_k)_{k\geq 1}$ is a sequence of real numbers such that,  for all $ k \geq 1, \alpha_k>-1$.
\end{description}

The next proposition states the existence of a solution of
(\ref{ImpulsiveEquation}) and its uniqueness.
\begin{proposition} Let hypotheses $(\mathbf{H_1})$ and $(\mathbf{H_2})$ hold true. Then there exists a unique global solution $X(\cdot;X_0,\alpha_k)$ of (\ref{ImpulsiveEquation}) and, for all $t\geq 0$, the application $(X_0,(\alpha_k)_{\{k\in \mathbb{N}^*, \tau_k\leq t\}})\mapsto X(t;X_0,\alpha_k)$ is continuous with respect to both variables.  Moreover, for $t\geq 0$,
\begin{equation}
X(t;X_0,\alpha_k)= X(t)=
X_0+\int_0^tg(X(s))\mathrm{d}s+\underset{\{ k\in \mathbb{N}^*,\ 0<\tau_k\leq t\}}{\sum}\alpha_k X(\tau_k^-).\label{FormeSol}
\end{equation} \label{ExistenceUnicité}
\end{proposition}
\begin{proof}
Existence and uniqueness of the solution, as well as  formula (\ref{FormeSol}), are given by Theorems 2.3 and 2.6 from \citet{Bainov1993}. 

The continuity of $X_0\mapsto X(t;\alpha_k,X_0)$ is proved in Theorem 1.2. from \citet{Dishliev2012} for a more general class of equations but under the hypothesis that there exists $C_g>0$ such that for all $x\in U$, $|g(x)|<C_g$. In the case of (\ref{ImpulsiveEquation}), this hypothesis is not necessary.  

Indeed, since the functions $X_0 \mapsto X(\cdot;X_0,0)$ and $x\mapsto (1+\alpha)x$, $ \alpha>-1$,  are continuous, it is easy to see that the solutions of (\ref{ImpulsiveEquation})  are continuous with respect to the initial condition $X_0$ and with respect to  $(\alpha_k)_{\{k\in \mathbb{N}^*,\ \tau_k\leq t\}}$. Note that $Card\left( \{k\in \mathbb{N}^*,\ \tau_k\leq t\}\right)=:N<+\infty$ then  any topology can be chosen for the continuity regarding parameter $(\alpha_k)_{\{k\in \mathbb{N}^*,\ \tau_k\leq t\}}\in \left( -1,+\infty\right) ^N$.\qed
\end{proof} 
\begin{lemma} Let hypothesis $\mathbf{(H_1)}$ hold true, $g(0)=0$, $U=\mathbb{R}^+$ and $X_0\in U$. If $(m_k)_{k\geq 1}$, $(\alpha_k)_{k\geq 1}$ and $(M_k)_{k\geq 1}$ are three sequences such that, for all $k\in \mathbb{N}^*$, $-1<m_k\leq \alpha_k\leq M_k$, then for all $t\in \mathbb{R}^+$,$$\ X(t;X_0,m_k)\leq X(t;X_0,\alpha_k)\leq X(t;X_0,M_k).$$ \label{encadrement}\end{lemma}
If $U=\mathbb{R}^-$ Lemma \ref{encadrement} remains true if we reverse the inequalities. If $U=\mathbb{R}$ it is necessary to discuss the sign of the solutions. 

\begin{remark} One specific feature of IDE is that two distinct solutions might merge after an impulse depending on the impulse function. In the case of (\ref{ImpulsiveEquation}), impulses are of the particular form $X(\tau_k^+)=(1+\alpha_k)X(\tau_k^-)$ and, since for all $\alpha > -1$ the function $X\mapsto (1+\alpha)X$ is injective and increasing, two distinct solutions cannot cross or merge after an impulse. Indeed, on intervals $ [0, \tau_1)$ and $(\tau_k, \tau_{k+1})$, the Cauchy-Lipschitz theorem ensures that distinct solutions cannot overlap. Consequently, under hypotheses ($\mathbf{H_1}$)-($\mathbf{H_2}$), for a given sequence  $(\alpha_k)_{k\geq 1}$ and two initial conditions $X_m<X_M$, for all $t\geq 0$, $X(t;X_m,\alpha_k)<X(t;X_M,\alpha_k)$. \end{remark}

Hereafter, we study the behaviour of the solutions of (\ref{ImpulsiveEquation}) when we consider that impulses occur with fixed period $\omega\in \mathbb{R^+_*}$ and that the sequence $(\alpha_k)_{k\in \mathbb{N}^*}$ is constant. To this end, we introduce hypotheses ($\mathbf{H_3}$) and ($\mathbf{H_4}$), as follows:

\begin{description}
\item[($\mathbf{H_3}$)] $\forall k \in \mathbb{N}^*,\ \tau_k=k\omega$ with $\omega>0$ fixed.
\item[($\mathbf{H_4}$)] $\forall k \in \mathbb{N}^*,\ \alpha_k=\alpha>-1$ with $\alpha$ fixed.
\end{description}

Note that ($\mathbf{H_3}$) implies ($\mathbf{H_1}$) and ($\mathbf{H_4}$) implies ($\mathbf{H_2}$).



We introduce the next definition.
\begin{definition}
A solution $X$ of (\ref{ImpulsiveEquation}) converges to a solution $Y$ if for all $\epsilon>0$, there exists $t_\epsilon >0$ such that $|X(t)-Y(t)|<\epsilon$  for $t>t_\epsilon$.
\end{definition}



The following proposition is a particular case of Theorem 12.5 from \citet{Bainov1993}.
\begin{proposition}
Let hypotheses $(\mathbf{H_3})$ and $(\mathbf{H_4})$ hold true. Then every bounded solution of (\ref{ImpulsiveEquation})  converges to a periodic solution.
\label{borneeconverge}
\end{proposition}

Under hypotheses ($\mathbf{H_3}$) and ($\mathbf{H_4}$), one can show Lemma \ref{monotonie} and Proposition~\ref{periode}
.

\begin{lemma} 
Let hypotheses $(\mathbf{H_3})$ and $(\mathbf{H_4})$ hold true and consider the solution $X=X(\cdot;X_0,\alpha)$  of (\ref{ImpulsiveEquation}). Then, \begin{description}
\item[i)] if $X(\omega)> X(0)$ (resp. $<X(0)$) the sequence $(X(k\omega))_{k\in\mathbb{N}^*}$ is strictly increasing (resp. decreasing).
\item[ii)] $X$ is $\omega$-periodic if and only if $X(\omega)=X(0)$ (Lemma 12.1 from \citet{Bainov1993}).
\end{description}  \label{monotonie}
\end{lemma}
\begin{proof}  We first show \textbf{i)}. Let us assume that $X(\omega):=(1+\alpha) X(\omega^-)>X(0)$. We then assume that there exists $k\in \mathbb{N}^*$ such that $X_{k+1}:=X((k+1)\omega)>X_k:=X(k\omega)$ and we show that $X((k+2)\omega)>X((k+1)\omega)$.

Thanks to the uniqueness of the solution of (\ref{ImpulsiveEquation}), and by integrating the solutions $Y_1=Y_1(\cdot;X_{k+1},\alpha)$ and $Y_2=Y_2(\cdot;X_k,\alpha)$ of (\ref{ImpulsiveEquation}) on $\left[0,\omega \right)$, the inequality $X((k+1)\omega)>X(k\omega)$ implies that $X((k+2)\omega^-)>X((k+1)\omega^-)$ and then $$X((k+2)\omega):=(1+\alpha)X((k+2)\omega^-)>(1+\alpha)X((k+1)\omega^-)=:X((k+1)\omega).$$ 

We showed that for all $k\in \mathbb{N}^*,\ X((k+1)\omega)>X(k\omega)$. Similarly, we can show that if $X(\omega)<X(0)$, then for all $k\in \mathbb{N}^*,\ X((k+1)\omega)<X(k\omega)$. 

It remains to prove \textbf{ii)}. As done in the proof of \textbf{i)} we show that if $X(\omega)=X(0)$, then for all $k\in \mathbb{N}^*,\ X(k\omega)=X(0)$. Because $X$ is the solution of an autonomous equation on each interval $\left[ k\omega, (k+1)\omega\right)$, it follows that for all $k\in \mathbb{N}^*$, for all $ t\in \left[0,\omega\right) $, $X(t)=X(k\omega+t)$ and then $X$ is $\omega$-periodic.  \end{proof}

\begin{proposition}
Let hypotheses $(\mathbf{H_3})$ and $(\mathbf{H_4})$ hold true and $X$ be a non-constant periodic solution of (\ref{ImpulsiveEquation}). Then $X$ is $\omega$-periodic and $\omega$ is the smallest period of $X$. \label{periode}\end{proposition}
\begin{proof} 
If $\alpha=0$ (i.e. no impulse) there is no non-constant periodic solution of (\ref{ImpulsiveEquation}) because the solution of an autonomous scalar differential equation is monotonous. Then, in the rest of the proof, we suppose $\alpha \neq 0$. 

We set $\alpha \in (-1,+\infty)\backslash \{0\}$ and let $X$ be a periodic solution of (\ref{ImpulsiveEquation}) with period $P>0$. It is easy to show that either $X\equiv 0$, or for all $t\geq 0,\ X(t)>0 $, or for all $t\geq 0,\ X(t)<0$. Indeed, the sign of $X$ changes at most one time between two impulses and since $\alpha>-1$, the solution $X$ cannot vanish or change its sign due to an impulse.  Because $X$ is supposed to be non-constant, for all $t\in \mathbb{R}^+,\ X(t) \neq 0 $. Moreover, $(1+\alpha)X(\omega^-)=X(\omega)\neq 0 $ then $X(\omega^-)\neq 0$. 

 We first show that the smallest period of $X$ is a multiple of $\omega$. Let us assume that $X$ is $P$-periodic with $\omega \nmid P$. Because of the $P$-periodicity of $X$, \begin{equation} X(\omega)=X(P+\omega) \text{ and } X(\omega^-)=X((P+\omega)^-).\label{egaliteLemmePeriode}\end{equation} Moreover, since $X(\omega^-)\neq 0$ and $\alpha>-1$, then  \begin{equation} X(\omega)=(1+\alpha)X(\omega^-)\neq X(\omega^-).\label{inegaliteLemmePeriode} \end{equation} 
 Consequently, using (\ref{egaliteLemmePeriode}) and (\ref{inegaliteLemmePeriode}), we have \begin{equation} X(P+\omega)\neq X((P+\omega)^-).\label{FirstHand}\end{equation}
On the other hand, since $\omega \nmid (P+\omega)$, $t=P+\omega$ is not an impulsion time, so we have $$X(P+\omega)=X((P+\omega)^-).$$ This contradicts (\ref{FirstHand}). Therefore, the smallest period of $X$ is a multiple of $\omega$. 

Now, it suffices to show that $X(\omega):=(1+\alpha) X(\omega^-)=X(0)$.
 Let us assume that ${X(\omega)\neq X(0)}$. According to Lemma \ref{monotonie}, $(X(k\omega))_{k\in\mathbb{N}^*}$ is a strictly monotonous sequence, so for all $k\in \mathbb{N}^*$, $X(k\omega)\neq X(0)$, which is absurd because $X$ is $P$-periodic with $\omega | P$. Finally, the solution $X$ is $\omega$-periodic.  
\end{proof}

In this section, we proved the existence of a solution for System (\ref{ImpulsiveEquation}) and established some properties for its behaviour.  According to Propositions \ref{borneeconverge} and \ref{periode}, either a solution of (\ref{ImpulsiveEquation}) converges to a $\omega$-periodic solution, or it diverges to $\pm \infty$. Consequently, it suffices to study the periodic solutions of (\ref{ImpulsiveEquation}) to conclude on the asymptotic behaviour of any bounded solution. That is the focus of the next section.

\section{Existence and stability of periodic solutions}\label{SectionSolPeriodiques}
Throughout this section, hypotheses ($\mathbf{H_3}$) and ($\mathbf{H_4}$) hold true. We present results on the existence and stability of periodic solutions of (\ref{ImpulsiveEquation}). 

 We introduce $\varphi_\omega$, the flow of System (\ref{ImpulsiveEquation}) at time $\omega$ and without impulsion, defined by \begin{equation}
\varphi_\omega:   X_0\in U \mapsto X(\omega;X_0,0)\in U. 
\label{defVarphi}
\end{equation} 
Note that, since the first impulse in System (\ref{ImpulsiveEquation}) occurs at time $\omega$, for any value of $\alpha$ we have  \begin{equation}
\varphi_\omega(X_0)= X(\omega;X_0,0)= X(\omega^-;X_0,\alpha)=X(\omega;X_0,\alpha)/(1+\alpha). \label{RelationPhi}
\end{equation} 
We can also notice that, thanks to the existence and uniqueness of the solution of  (\ref{ImpulsiveEquation}), $\varphi_\omega$ is strictly increasing on $U$ and it can easily be shown that $\varphi_\omega$ is a bijection from $U$ to $U$.  In addition, if $g\in\mathcal{C}^p(U)$ with $p\in \mathbb{N}^*$, then $\varphi_\omega\in\mathcal{C}^p(U)$ (Cauchy-Lipschitz theorem) where $\mathcal{C}^p(U)=\{h:U\to\mathbb{R},\ h^{(p)}\text{ exists and is continuous}\}$.

For the sake of simplicity, we also introduce the function $\Gamma_\omega$ defined by \begin{equation}
\begin{array}{l r l}
\Gamma_\omega:&  X_0\in U\backslash\{\varphi_\omega^{-1}(0) \} &\mapsto \dfrac{X_0}{\varphi_\omega(X_0)}-1\in \mathbb{R}.
\end{array}\label{DefGamma}
\end{equation}
\begin{remark}
As a ratio of two continuous functions, the function $\Gamma_\omega$ is continuous everywhere $\varphi_\omega$ is non-zero, that is on $U\backslash\{\varphi_\omega^{-1}(0) \}$.
 In particular, if $g(0)=0$ and $U= \mathbb{R^+}$ (resp. $U=\mathbb{R^-}$), then $\varphi_\omega^{-1}(0)=0$ and ${\Gamma_\omega : \mathbb{R}^+_* \to (-1,+\infty)  }$ (resp. $\Gamma_\omega : \mathbb{R}^-_* \to (-1,+\infty)  $ ).
\label{ContinuitéGamma}\end{remark}

For a given initial condition $X_0$ and under hypotheses ($\mathbf{H_3}$) and ($\mathbf{H_4}$), if one knows the value of $\varphi_\omega(X_0)$ then one can conclude on the existence and the value of $\alpha$ for which the solution $X(\cdot;X_0,\alpha)$ of (\ref{ImpulsiveEquation}) is periodic. This is stated in the next proposition. 

\begin{proposition} Let hypotheses $(\mathbf{H_3})$ and $(\mathbf{H_4})$ hold true. For any initial condition $X_0\in U$ such that $\varphi_\omega(X_0)\neq 0$, the solution $X=X(\cdot;X_0,\alpha)$ of (\ref{ImpulsiveEquation})  is periodic if and only if $$\alpha=\Gamma_\omega (X_0).$$
Moreover, if $\varphi_\omega(X_0)=0$, either $g(0)=0$, and then for any $\alpha \in \mathbb{R}$ the solution $X(\cdot;X_0,\alpha)$ of (\ref{ImpulsiveEquation}) satisfies $X\equiv 0$, in particular $X$ is periodic ; or  $g(0)\neq 0$, and then for any choice of $\alpha>-1$, the solution $X(\cdot;X_0,\alpha)$ of (\ref{ImpulsiveEquation}) is not periodic. 
\label{PropGamma}
\end{proposition} 

 \begin{proof}
From Lemma \ref{monotonie} and Proposition \ref{periode}, $X=X(\cdot;X_0,\alpha)$ is periodic if and only if $X(\omega)=X_0$, that is, from (\ref{RelationPhi}), $(1+\alpha)\varphi_\omega(X_0)=X_0$. 
If $\varphi_\omega (X_0) \neq 0$, $X$ is periodic if and only if  $(1+\alpha)\varphi_\omega(X_0)=X_0$, that is, from (\ref{DefGamma}), $\alpha=\Gamma_\omega (X_0)$.
In the remainder of the proof we assume $\varphi_\omega(X_0)=0$.

If $g(0)=0$, then $X\equiv 0$. Indeed $Y\equiv 0$ is a solution of (\ref{ImpulsiveEquation}) and, for any $\alpha\in \mathbb{R}$, from (\ref{RelationPhi}), $X(\omega;X_0,\alpha)=(1+\alpha)\varphi_\omega (X_0)=0=Y(\omega)$. By the uniqueness of the solution, for any $\alpha\in \mathbb{R}$, $X=Y\equiv 0$, in particular $X$ is periodic.

On the other hand, if $g(0)\neq 0$, then $X(0)\neq 0$. Indeed, if $X(0)=0$, then $X$ is strictly monotonous on $\left[0,\omega\right)$ and $\varphi_\omega(X_0)\neq 0$, hence a contradiction. Moreover, for any $\alpha\in \mathbb{R}$, $X(\omega)=(1+\alpha)\varphi_\omega (X_0)=0$. Consequently $X(\omega)\neq X(0)$ and thanks to Proposition \ref{periode} $X$ is not periodic.  
\end{proof}

In the remainder of this section, we study the stability of the periodic solutions of (\ref{ImpulsiveEquation}).  For the sake of simplicity, for any $\alpha>-1$ and $\omega>0$ we introduce the function $\gamma_{\omega,\alpha}$ defined by
\begin{equation}
\gamma_{\omega,\alpha}: X_0\in U \mapsto \varphi_\omega (X_0)-\frac{X_0}{1+\alpha}\in U.
\label{defgammaalpha}
\end{equation}
\begin{lemma}
Let hypotheses $(\mathbf{H_3})$ and $(\mathbf{H_4})$ hold true and let $X_0\in U$. Let $K^+:=\gamma_{\omega,\alpha}^{-1}(\{0\})\cap (X_0,+\infty)$ and $K^-:=\gamma_{\omega,\alpha}^{-1}(\{0\})\cap (-\infty,X_0)$.  Then \begin{enumerate}
\item[i)] if $\gamma_{\omega,\alpha} (X_0)>0$, either $K^+ \neq \emptyset$ and then $X(\ \cdot \ ; X_0, \alpha)$ converges to the periodic solution $Y(\ \cdot \ ;Y^+,\alpha)$ where $Y^+=\min (K^+)$, or $K^+=\emptyset$ and $X(\ \cdot \ ; X_0, \alpha)$ diverges to  $+\infty$;
\item[ii)] if $\gamma_{\omega,\alpha} (X_0)<0$ , either $K^-\neq \emptyset$ and then $X(\ \cdot \ ; X_0, \alpha)$ converges to the periodic solution $Y(\ \cdot \ ;Y^-,\alpha)$ where $Y^-=\max (K^-)$, or $K^-=\emptyset$ and $X(\ \cdot \ ; X_0, \alpha)$ diverges to  $-\infty$;
\item[iii)] if $\gamma_{\omega,\alpha} (X_0)=0$, then $X(\ \cdot \ ; X_0, \alpha)$ is $\omega$-periodic.
\end{enumerate}
\label{monotonieselongamma}
\end{lemma}
\begin{proof}
Proofs of points \textit{i)} and \textit{ii)} are similar, hence we only show point \textit{i)}.

We first assume that $\gamma_{\omega,\alpha }(X_0)>0$ and show that, when $K^+\neq \emptyset$, $Y^+$ is well defined. Let us assume that $K^+\neq \emptyset$. Since  $\gamma_{\omega,\alpha}$ is a continuous function and $\{0\}$ is a closed set, then $\gamma_{\omega,\alpha}^{-1}(\{0\})$ is a closed set. Moreover, since $\gamma_{\omega,\alpha }  (X_0)>0$, then $X_0\notin \gamma_{\omega,\alpha}^{-1}(\{0\})$ and we can write $K^+=\gamma_{\omega,\alpha}^{-1}(\{0\})\cap [X_0,+\infty)$. Finally, $K^+$ is a non-empty closed set which admits a lower bound ($X_0$), then it admits a minimum $Y^+>X_0$.
 
 We now prove \textit{i)}. Hereafter, $X=X(\ \cdot \ ;X_0,\alpha)$. We notice that inequality $\gamma_{\omega,\alpha } (X_0)>0$ is equivalent to $X(\omega)>X(0)$. Then, according to Lemma \ref{monotonie} and Proposition \ref{borneeconverge}, the sequence $(X(k\omega))_{k\in \mathbb{N} }$ is (strictly) increasing and either $X$ diverges to $+\infty$, or it converges to a periodic solution. Consequently, if there is no periodic solution with initial condition $K_0>X_0$ (i.e. $K^+= \emptyset$), then $X$ diverges to $+\infty$. On the other hand, if $K^+\neq \emptyset$, then we have for all $t\geq 0,\ X(t) \leq Y(t;Y^+,\alpha)$, so $X$ is bounded and converges to a periodic solution of (\ref{ImpulsiveEquation}). From the definition of $Y^+$ and since $X(t) \leq Y(t)$ for all $t\geq 0$, we can conclude that $X$ converges to $Y$. 
 
Point \textit{iii)} is equivalent to Lemma  \ref{monotonie}. Indeed, using (\ref{RelationPhi}), $X$ is periodic if and only if $X(\omega)=(1+\alpha)\varphi_\omega (X_0)=X_0$, that is $\gamma_{\omega,\alpha } (X_0)=0$. 
\end{proof}

Hereafter, we adapt the usual notions of stability used for non-impulsive differential equations to our problem, as follows. 

\begin{definition}
Let hypotheses $(\mathbf{H_3})$ and $(\mathbf{H_4})$ be satisfied. Let $X_*\in U$ be such that $\Gamma_\omega (X_*)=\alpha$ and $I\subset U$ be an interval. Then, the periodic solution $X(\ \cdot \ ; X_*,\alpha)$ is said to be\begin{enumerate}
\item[i)]stable if for any neighbourhood $V$ of $X_*$, there is a neighbourhood $W$ of $X_*$ such that for all $X_0\in W$ and for all $k\in \mathbb{N}, \ X(k\omega,X_0,\alpha)\in V$;
\item[ii)] asymptotically stable if it is stable and there exists a neighbourhood $V$ of $X_*$ such that for any $X_0\in V$,  $\underset{k\to +\infty}{\lim} X(k\omega ;X_0,\alpha) = X_*$;
\item[iii)] unstable if it is not stable;
\item[iv)] attractive on $I$ if for any $X_0\in I$, the sequence $\left( |X(k\omega;X_0,\alpha)-X_*|\right)_{k\in \mathbb{N}}$ decreases and converges to zero;
\item[v)] repulsive on $I$ if for any $X_0\in I$, the sequence $\left( |X(k\omega;X_0,\alpha)-X_*|\right)_{k\in \mathbb{N}}$ increases.
\end{enumerate}
\end{definition}

\begin{remark}
As an immediate consequence of Lemma \ref{monotonieselongamma}, under hypotheses ($\mathbf{H_3}$) and ($\mathbf{H_4}$) the stability of a periodic solution $X(\ \cdot \ ;X_0,\alpha)$ is given by the sign of $\gamma_{\omega,\alpha}$ on both sides of $X_0$. The stability of periodic solutions of (\ref{ImpulsiveEquation}) can however also be sometimes deduced from the graph of $\Gamma_\omega$. Indeed, let $X_0\in U$ be such that $\varphi(X_0)> 0$, then, from (\ref{DefGamma}) and (\ref{defgammaalpha}), $\gamma_{\omega,\alpha} (X_0)>0$ (resp. $\gamma_{\omega,\alpha} (X_0)<0$, resp. $\gamma_{\omega,\alpha} (X_0)=0$) if and only if $\Gamma_\omega (X_0) < \alpha$ (resp. $\Gamma_\omega (X_0) > \alpha$, resp. $\Gamma_\omega (X_0) = \alpha$)
and we can conclude on the asymptotic behaviour of $X(\ \cdot \ ; X_0, \alpha)$ (Lemma \ref{monotonieselongamma}) based on the graph of $\Gamma_\omega$. For example, let $U=\mathbb{R}^+$ and $X_0>0$ be such that $\Gamma_\omega(X_0)=\alpha$. If $\Gamma_\omega$ is strictly increasing (resp. decreasing) on a neighbourhood of $X_0$, then $X(\ \cdot \ ;X_0,\alpha)$ is asymptotically stable (resp. unstable). 
\label{LireStabilité}
\end{remark}

From Proposition \ref{borneeconverge}, under ($\mathbf{H_3}$) and ($\mathbf{H_4}$), it suffices to determine the initial values of the periodic solutions of System (\ref{ImpulsiveEquation}) and their stability to conclude on the asymptotic behaviour of all the solutions.  In Sections \ref{PropGénérales} and \ref{SectionSolPeriodiques}, we have shown that, under hypotheses ($\mathbf{H_3}$) and ($\mathbf{H_4}$), for any $X_0\in U$ such that $\varphi_\omega(X_0)\neq 0$ there exists a unique $\alpha \in \mathbb{R}$ such that $X(\ \cdot \ ;X_0,\alpha)$ is a periodic solution of (\ref{ImpulsiveEquation}) (Proposition \ref{PropGamma}). 
 Conversely, for any given $\alpha>-1$, we discussed the existence of $X_0=X_0(\alpha) \in U$ such that the solution $X=X(\cdot;X_0,\alpha)$ of (\ref{ImpulsiveEquation}) is periodic.

As a consequence of Lemma \ref{monotonieselongamma}, the initial conditions of the periodic solutions of (\ref{ImpulsiveEquation}) correspond to the solutions $X_0$ of equation $\gamma_{\omega,\alpha}(X_0)=0$, with $\gamma_{\omega,\alpha}$  defined by  (\ref{defgammaalpha}). Necessary and sufficient conditions to determine the stability of the periodic solutions are given in Lemma \ref{monotonieselongamma} and Remark \ref{LireStabilité}.

 However, in general, the expression of $\varphi_\omega$, defined by (\ref{defVarphi}), is not known and the exact values of the initial conditions of the periodic solutions of (\ref{ImpulsiveEquation}) cannot be determined. In that case \textit{ad hoc} studies have to be performed. In Section \ref{sectionConvex}, we give a sufficient condition for $\varphi_\omega$ to be convex. This result will be used in Sections \ref{ChapitreImpulseTbet} and \ref{SectionBio}.

\section{Convexity of the flow of an autonomous equation} \label{sectionConvex}
In this section, we aim at investigating the convexity of the flow $\varphi_\omega$, defined by (\ref{defVarphi}), which can further be  useful to determine the number of periodic solutions of (\ref{ImpulsiveEquation}) when ($\mathbf{H_3}$) and ($\mathbf{H_4}$) hold true (Proposition \ref{PropGamma}). In Section \ref{Sectionhn} we study a Cauchy problem where the right-hand side of the equation is a piecewise linear function and we extrapolate this result to continuously differentiable functions in Section \ref{Sectionh}.

\subsection{Preliminary results}\label{Sectionhn}
We set $a<b\in \mathbb{R}$, $n\in \mathbb{N^*}$ and $(c_i)_{0\leq i \leq n}$ a subdivision of $[a,b]$ such that $c_0=a,\ c_n=b$ and, for all $i\in \{ 0,\dots, n-1 \}$, $c_i<c_{i+1}$.  Then, we define the piecewise linear function $h_n$, on $[a,b]$, by \begin{equation}
\begin{array}{l l}
h_n  (x)= a_{i+1}x+b_{i+1} & \text{ for } x\in [c_i,c_{i+1}),\ i\in \{0,\dots ,n-1\},
\end{array}\label{Def_fn}\nonumber
\end{equation}
where real numbers $(a_i)_{1\leq i \leq n}$, $(b_i)_{1\leq i \leq n}$ are such that $h_n$ is positive and continuous, that is \begin{equation}
h_n(c_0)>0\text{ and }\forall 1\leq i \leq n-1,\ a_ic_i+b_i=a_{i+1}c_i+b_{i+1}=h_n(c_i)>0.\label{Hyp_hn}
\end{equation}

Then, we introduce the autonomous Cauchy problem 
\begin{equation}
\left \{ \begin{array}{r c l} 
X'(t)&=&h_n(X(t)),\ t\geq 0,\\
X(0)&=&X_0.
\end{array}\right. \label{System_fn}
\end{equation} 
For the sake of simplicity, we set $h_n(c_n)=0$ such that for any $X_0\in [a,b)$, the solution of (\ref{System_fn}) is increasing and remains in $[a,b)$, and $X=c_n=b$ is the unique steady state of (\ref{System_fn}).

Note that existence and uniqueness for the solution $X$ of System (\ref{System_fn}) is guaranteed by Cauchy-Lipschitz theorem, most of the time this solution will be denoted by $X_n$. In particular, the restriction of $X_n$ on each interval $[c_{i-1},c_i]$, $i\in \{ 1,..., n-1 \}$,  is solution of a linear equation.

 In the following, we first determine the values of $t$ such that $X_n(t)=c_i$ for $i=0,...,n-1$ (Lemma \ref{PropLambdak}) and then we provide an explicit expression for $X_n$ (Proposition \ref{PropSolXn}).

Since $h_n(x)>0$ for all $x\in [c_0,c_n)$ and $h_n(c_n)=0$, it is clear that for all $X_0\in (c_0,c_n)$, $X_n$ is an increasing function and $\lim_{t\to +\infty} X_n(t)=c_n$. 
Straightforward calculations show that, if $X_0\in [c_{i},c_{i+1})$ for $i\in \{0,..., n-2 \}$, then $X_n(t)=c_{i+1}$ if and only if $t=\lambda_{i+1}(X_0)$, where $\lambda_{i+1}$ is defined for all $x\in [c_i,c_{i+1})$ by
\begin{equation}\lambda_{i+1}(x)= \begin{cases}
  \frac{1}{a_{i+1}}\ln \left(\frac{h_n(c_{i+1})}{h_n(x)} \right), & \mbox{if } a_{i+1} \neq 0, \\
  \frac{c_{i+1}-x}{b_{i+1}},  & \mbox{if } a_{i+1}= 0.
\end{cases} \label{FormuleLambdap}\nonumber \end{equation}
By induction, we deduce the expression of $\lambda_k(X_0)$ for any $X_0<c_k$.
\begin{lemma}
Let $i\in \{ 0,\dots, n-2 \}$, $X_0\in [c_i,c_{i+1})$, and $X_n$ be the solution of (\ref{System_fn}). 
Let  $k\in \{i+1,...,n-1 \}$, then $X_n(t)=c_k$ if and only if $t=\lambda_k(X_0)$, defined by \begin{equation}
\lambda_k(X_0)=\lambda_{i+1}(X_0)+\sum_{p=i+1}^{k-1}\lambda_{p+1}(c_{p}).\label{ExpressionLambdak}
\end{equation}  \label{PropLambdak}
\end{lemma}

In other words, $\lambda_k(X_0)$ corresponds to the time for the solution $X_n$ of (\ref{System_fn}) to go from $X_0$ to $c_k$. Since $c_n=\lim_{t\to +\infty}X_n(t)$, we set $\lambda_n(X_0)=+\infty$. If $X_0\geq c_p$, for $p\in \{0,\dots , n-1\}$, we set $\lambda_p(X_0)=0$. 

\begin{proposition}
Let $i \in \{ 0,...,n-1 \}$, $X_0\in [c_i,c_{i+1})$ and $X_n$ be the solution of (\ref{System_fn}). Then, if $t\in [0,\lambda_{i+1}(X_0))$,
\begin{equation} X_n(t)= \begin{cases}
  \frac{h_n(X_0)}{a_{i+1}}e^{a_{i+1}t}-\frac{b_{i+1}}{a_{i+1}}, & \mbox{if } a_{i+1} \neq 0, \\
  X_0+b_{i+1}t,  & \mbox{if } a_{i+1}= 0. \\ 
\end{cases} \nonumber \label{XnCi}
\end{equation}
If $t\in \left[\lambda_k(X_0),\lambda_{k+1}(X_0)\right)$ for $k\in \{i+1,...,n-1\}$, then
\begin{equation}
X_n(t)= \begin{cases}
\left( c_k-\frac{b_{k+1}}{a_{k+1}}\right)e^{a_{k+1}(t-\lambda_k(X_0))}-\frac{b_{k+1}}{a_{k+1}}, & \mbox{if } a_{k+1} \neq 0, \\
c_k+b_{k+1}\left( t-\lambda_k(X_0)\right), & \mbox{if } a_{k+1} = 0.
\end{cases}\label{XnCk}
\end{equation}  \label{PropSolXn}  
\end{proposition}

Since, from Lemma \ref{PropLambdak}, $X_n(\lambda_k(X_0))=c_k$, then to prove Proposition \ref{PropSolXn} one only has to solve the linear differential equation $X'(t)=a_{k+1}X(t)+b_{k+1}$ .

We now consider $\omega>0$ and $\varphi_{\omega}^n$ , the flow of (\ref{System_fn}) at time $\omega$, defined by\begin{equation}
\begin{array}{l r l l l}
\varphi_{\omega}^n :&  X_0\in [a,b] &\mapsto X_n(\omega)\in [a,b],\end{array}\label{Def_phin}
\end{equation}
where $X_n$ is the solution of (\ref{System_fn}).

  Note that, since $h_n$ is continuous, $\varphi_{\omega}^n$ is also continuous.   If we set $t=\omega$, Proposition \ref{PropSolXn} provides an explicit expression for $\varphi_{\omega}^n(X_0)$. 
In the following, we determine the convexity of $\varphi_\omega^n$ by studying its second derivative. However, if $X_0=c_k$ or $\varphi_\omega^n(X_0)=c_k$, the expression of $\varphi_\omega^n$ is not the same on both sides of $X_0$ and, in that case, we cannot directly compute the derivative of $\varphi_\omega^n$ . Consequently, we first study the convexity of $\varphi_\omega^n$ everywhere $X_0\neq c_k$ and $\varphi_\omega^n(X_0)\neq c_k$, and then we conclude on the convexity of $\varphi_\omega^n$ on the whole interval $(a,b)$. For the sake of simplicity we define \begin{equation}J_n=\{c_k,\ k=1,\dots,n-1 \} \cup \{(\varphi_\omega^n)^{-1}(c_k),\ k=1,\dots,n-1 \}.\label{defJn} \end{equation}
\begin{lemma}
Let $X_0 \in (a,b)\backslash J_n$ and $i,k\in \{0,\dots, n-1\}$ such that $X_0\in (c_i,c_{i+1})$ and $\varphi_{\omega}^n(X_0)\in (c_k,c_{k+1})$. 
 Then, on a neighbourhood of $X_0$, $\varphi_{\omega}^n$ is $\mathcal{C}^\infty$  and strictly convex (resp. strictly concave, resp. affine) if and only if $a_{k+1}>a_{i+1}$ (resp. $a_{k+1}<a_{i+1}$, resp. $a_{k+1}=a_{i+1}$).
\label{LemmeConvexité}
\end{lemma}

\begin{proof}
To prove Lemma \ref{LemmeConvexité}, it suffices to study the sign of the second derivative of $\varphi_\omega^n$.  Depending on whether $a_{i+1}$ and $a_{k+1}$ are null are not, $\varphi_\omega^n$ can be defined by four different expressions. Here we only provide a proof for the case $a_{i+1}\neq 0$ and $a_{k+1}\neq 0$. Other cases are analogous. 

Let $X_0\in (c_i,c_{i+1})$ such that $\varphi_\omega^n(X_0)\in (c_k,c_{k+1})$, $k>i$. Then, according to equations (\ref{ExpressionLambdak}) and (\ref{XnCk}) \begin{equation}
\begin{array}{l l}
\varphi_\omega^n(X_0)&=\left( c_k-\frac{b_{k+1}}{a_{k+1}}\right)e^{a_{k+1}(\omega-\lambda_k(X_0))}-\frac{b_{k+1}}{a_{k+1}},\vspace{0.3cm} \\ 
&=\frac{h_n(c_k)}{a_{k+1}}e^{a_{k+1}\omega}\prod_{p=i+1}^{k-1} e^{-a_{k+1}\lambda_{p+1}(c_p)}e^{-a_{k+1}\lambda_{i+1}(X_0)} \vspace{0.3cm} \\ 
&=\frac{h_n(c_k)}{a_{k+1}}h_n(X_0)^\frac{a_{k+1}}{a_{i+1}}{M_{i,k}}^{a_{k+1}}-\frac{b_{k+1}}{a_{k+1}},
\end{array}\label{VarphiAvecM}
\end{equation}
where \begin{equation} M_{i,k}=e^{\omega}\left( \frac{1}{h_n(c_{i+1})}\right)^\frac{1}{a_{i+1}}\prod_{p=i+1}^{k-1}\left(\frac{h_n(c_{p})}{h_n(c_{p+1})}\right)^\frac{1}{a_{p+1}}>0.\nonumber\end{equation}
Consequently \begin{equation}
\begin{array}{l l}
\displaystyle{\frac{\mathrm{d}^2\varphi_\omega^{n} }{\mathrm{d}X_0^2}}(X_0)&=\frac{h_n(c_k){M_{i,k}}^{a_{k+1}}}{a_{k+1}}\left(\frac{a_{k+1}}{a_{i+1}}\right)\left(\frac{a_{k+1}}{a_{i+1}}-1\right)a_{i+1}^2\left(a_{i+1}X_0+b_{i+1}\right)^{\left(\frac{a_{k+1}}{a_{i+1}}-2\right)}\\
&= h_n(c_k){M_{i,k}}^{a_{k+1}}(a_{k+1}-a_{i+1})\left(a_{i+1}X_0+b_{i+1}\right)^{\left(\frac{a_{k+1}}{a_{i+1}}-2\right)}.
\end{array}\nonumber
\end{equation}
Moreover, on any neighbourhood of $X_0$ included in $(a,b)\backslash J_n$, $\varphi_\omega^n$ is linear and then its second derivative is constant. 
Finally, the convexity of $\varphi_\omega^n$ is determined by the sign of its second derivative, given by the sign of $(a_{k+1}-a_{i+1})$, as stated in Lemma \ref{LemmeConvexité}.  
\end{proof}
Thanks to Lemma \ref{LemmeConvexité}, we can determine the convexity of $\varphi_{\omega}^n$ on each interval included in $(a,b)\backslash J_n$. To extend this result to the whole interval $(a,b)$ we need to introduce the following lemma.
\begin{lemma}
$\varphi_{\omega}^n$ is continuously differentiable on $(a,b)$.
\label{LemmeRegularitePhi}
\end{lemma}
\begin{proof}
We know that $\varphi_\omega^n$ is $\mathcal{C}^\infty$ on $(a,b)\backslash J_n$. Let $X_0\in J_n$ and $i,k\in \{0,\dots,n-1\}$ such that $X_0=c_i$ or $\varphi_\omega^n(X_0)=c_k$. We want to show that the derivative of $\varphi_\omega^n$ is continuous in $X_0$. For the sake of simplicity, we only show the proof for the case where $a_i,\ a_{i+1},a_k$ and $\ a_{k+1}$ are non null, $X_0=c_i$ and $\varphi_\omega^n(X_0)=c_k$. Other cases are analogous.

Let $X_1\in (c_{i-1},c_i)$ such that $\varphi_\omega^n(X_1)\in (c_{k-1},c_k)$ and $X_2\in (c_{i},c_{i+1})$ such that $\varphi_\omega^n(X_2)\in (c_{k},c_{k+1})$

According to (\ref{VarphiAvecM}),
\begin{equation} \varphi_\omega^n(X_2)=\frac{h_n(c_k)}{a_{k+1}}h_n(X_2)^\frac{a_{k+1}}{a_{i+1}}{M_{i,k}}^{a_{k+1}}-\frac{b_{k+1}}{a_{k+1}}\nonumber \end{equation}
and straightforward calculations lead to 
\begin{equation}
{\renewcommand{\arraystretch}{1.6}
\begin{array}{l l}
\varphi_\omega^n(X_1)&=\frac{h_n(c_{k-1})}{a_{k}}h_n(X_1)^\frac{a_{k}}{a_{i}}{M_{i-1,k-1}}^{a_{k}}-\frac{b_{k}}{a_{k}}\\ &= \frac{h_n(c_{k})}{a_{k}}h_n(X_1)^\frac{a_{k}}{a_{i}} {M_{i,k}}^{a_{k}}\left(h_n(X_0)\right)^{\frac{a_k}{a_{i+1}}-\frac{a_k}{a_i}}-\frac{b_{k}}{a_{k}}.
\end{array}}\nonumber
\end{equation}
Consequently 
\begin{equation} {(\varphi_\omega^n})'(X_2)=h_n(c_k)h_n(X_2)^{\frac{a_{k+1}}{a_{i+1}}-1}{M_{i,k}}^{a_{k+1}}\label{Phi'2} \end{equation}
and 
\begin{equation} {(\varphi_\omega^n})'(X_1)= h_n(c_{k})h_n(X_1)^{\frac{a_{k}}{a_{i}}-1} {M_{i,k}}^{a_{k}}\left(h_n(X_0)\right)^{\frac{a_k}{a_{i+1}}-\frac{a_k}{a_i}}.\label{Phi'1} \end{equation}
On the other hand, since $\varphi_\omega^n$ is continuous and $\varphi_\omega^n(X_0)=c_k$, it follows that 
\begin{equation}
\begin{array}{l l}
c_k&=\varphi_\omega^n(X_0)=\underset{X_2\to X_0}{\lim }\varphi_\omega^n(X_2)=\frac{h_n(c_k)}{a_{k+1}}h_n(X_0)^\frac{a_{k+1}}{a_{i+1}}{M_{i,k}}^{a_{k+1}}-\frac{b_{k+1}}{a_{k+1}}.
\end{array} \nonumber \end{equation}
That is 
\begin{equation}
\begin{array}{l l}
a_{k+1}c_k+b_{k+1}=h_n(c_k) \left( h_n(X_0)^\frac{1}{a_{i+1}}{M_{i,k}}\right)^{a_{k+1}}.
\end{array} \label{EgaliteMik} \end{equation}
According to (\ref{Hyp_hn}), $h_n(c_k)=a_{k+1}c_k+b_{k+1}\neq 0$. It follows that $$h_n(X_0)^\frac{1}{a_{i+1}}M_{i,k}=1.$$

Applying (\ref{EgaliteMik}) to (\ref{Phi'2}) and (\ref{Phi'1}), we get 
\begin{equation} \begin{array}{l l} \underset{X_2\to X_0^+}{\lim }{\varphi_\omega^n}'(X_2)=\frac{h_n(c_k)}{h_n(X_0)}\left(h_n(X_0)^{\frac{1}{a_{i+1}}}{M_{i,k}}\right)^{a_{k+1}}=\frac{h_n(c_k)}{h_n(X_0)},\\ 
\underset{X_1\to X_0^-}{\lim }{\varphi_\omega^n}'(X_1)=\frac{h_n(c_k)}{h_n(X_0)}\left( h_n(X_0)^{\frac{1}{a_{i+1}}} {M_{i,k}}\right)^{a_{k}}\left(h_n(X_0)\right)^{\frac{a_k}{a_{i}}-\frac{a_k}{a_i}}= \frac{h_n(c_k)}{h_n(X_0)},
\end{array}\nonumber\end{equation}
and the proof is achieved.  \end{proof}

Consequently, the convexity of $\varphi_{\omega}^n$ on $(a,b)$ can be entirely determined, as stated in Proposition \ref{PropConvexitéPhin} hereafter.

Let us note, for all $n\in \mathbb{N}^*$, $h_n^*$ and $\sideset{^{*}}{_n}{\mathop{h}}$ the right and left derivatives of $h_n$ respectively. Then, $h_n^*$ and $\sideset{^{*}}{_n}{\mathop{h}}$ are defined everywhere in $(a,b)$ and given by \begin{equation} \left \{ \begin{array}{l l}
 \sideset{^{*}}{_n}{\mathop{h}}(x)= h_n^*(x)=h_n'(x)=a_{i+1}, &\text{ if } x\in (c_{i},c_{i+1}),\\ \sideset{^{*}}{_n}{\mathop{h}}(x)=a_{i},\   h_n^*(x)=a_{i+1}, &\text{ if } x=c_i.
\end{array} \right. \label{Valeurshn'} \end{equation}
\begin{proposition}
Let $X_0\in (a,b)$. Then $\varphi_{\omega}^n$ is strictly convex (resp. strictly concave, resp. affine) on a neighbourhood of $X_0$ if and only if $h_n^*(X_0)-h_n^*(\varphi_{\omega}^n(X_0))<0$ (resp. $>0$, resp. $=0$) and $\sideset{^{*}}{_n}{\mathop{h}}(X_0)-\sideset{^{*}}{_n}{\mathop{h}}(\varphi_{\omega}^n(X_0))<0$ (resp. $>0$, resp. $=0$).
\label{PropConvexitéPhin}
\end{proposition}

Proposition \ref{PropConvexitéPhin} is an immediate consequence of Lemmas \ref{LemmeConvexité} and \ref{LemmeRegularitePhi}. Indeed, Lemma \ref{LemmeConvexité} provides necessary and sufficient conditions for $\varphi_\omega^n$ to be strictly convex (resp. concave) on any interval included in $(a,b)\backslash J_n$, with $J_n$ defined by (\ref{defJn}), and Lemma \ref{LemmeRegularitePhi} ensures that $\varphi_\omega^n$ in strictly convex (resp. concave) on a neighbourhood of $X_0\in J_n$ if and only if it is strictly convex (resp. concave) on both sides of $X_0$.

\subsection{Generalisation to smooth functions} \label{Sectionh}
In this section, we consider the autonomous Cauchy problem \begin{equation}
\left \{ \begin{array}{r c l} 
X'(t)&=&h(X(t)),\ t\geq 0,\\
X(0)&=&X_0\in [a,b],
\end{array}\right. \label{System_h}
\end{equation}
where $h\in \mathcal{C}^1([a,b],\mathbb{R})$ is a Lipschitz continuous function, positive on $(a,b)$ and satisfying $h(b)=0$. 

We want to study the convexity of  $\varphi_\omega$, the flow associated to System (\ref{System_h}) at time $\omega>0$, defined by (\ref{defVarphi}). Remark that the hypothesis $h(b)=0$ ensures that for any $X_0\in [a,b]$, $X$ and $\varphi_ \omega$ remain in $[a,b]$.

The main purpose of this section is to prove the following theorem.
\begin{theorem}
Let $X_0\in (a,b)$. If $h'(\varphi_\omega(X_0))-h'(X_0)>0$ (resp. $<0$), then $\varphi_\omega$ is convex (resp. concave) on a neighbourhood of $X_0$. \label{TheoremConvexite}
\end{theorem} 
\begin{remark} Theorem~\ref{TheoremConvexite} can be directly adapted to the case where $h$ is a negative function and $h(a)=0$ by studying the function $g$ defined on $(-b,-a)$ by $g(x)=-h(-x)$. \label{Cashnegative}\end{remark}

In the following, we consider $h_n$ defined in Section \ref{Sectionhn} to be the order $n$ linear interpolation of $h$ on $[a,b]$ with uniform distribution of the nodes. That is, with the notations of Section \ref{Sectionhn},\begin{equation}
{\renewcommand{\arraystretch}{1}
\left\{ \begin{array}{l l}
 \forall i \in \{ 0,\dots, n \}, & c_i=a+i\frac{b-a}{n},\\  \forall i \in \{ 0,\dots, n-1 \},& a_{i+1}=\frac{h(c_{i+1})-h(c_i)}{c_{i+1}-c_i},\ b_{i+1}=h(c_i)-a_{i+1}c_i.
\end{array}\right. } \label{Valeurs_ai_bi_ci}
\end{equation} 
In particular, $h_n$ is continuous,
 for all $i \in \{ 0,\dots , n \},\ c_{i+1}-c_i=1/n$ and $h_n(c_i)=h(c_i)$.    When there is ambiguity on the order $n$, we note $a_i=a^{(n)}_{i}$, $b_i=b_{i}^{(n)}$ and $c_i=c_{i}^{(n)}$. The flow $\varphi_\omega^n$ associated with $h_n$ is defined in (\ref{Def_phin}).

\begin{lemma}
Let $X_0\in[a,b]$, then $\lim_{n\to +\infty}\varphi_{\omega}^n(X_0)=\varphi_\omega(X_0)$.\label{ConvergenceFlot}
\end{lemma}
\begin{proof}[Proof of Lemma \ref{ConvergenceFlot}.]
Let $X_0\in [a,b]$ and $X_n$, $X$ be the respective solutions of (\ref{System_fn}) and (\ref{System_h}). It is easy to show that for any $\epsilon>0$, there exists $n_\epsilon\in \mathbb{N}^*$ such that for all $n\geq n_\epsilon$ and for all $x\in [a,b]$, $|h(x)-h_n(x)|<\epsilon$. Then for all $x\in [a,b]$, $h(x)-\epsilon\leq h_n(x)\leq  h(x)+\epsilon$. 

For any $\epsilon>0$, let us consider $X_\epsilon$ and $X_{-\epsilon}$ the respective solutions of  $X_\epsilon'=h(X_\epsilon)+\epsilon$ and $X_{-\epsilon}'=h(X_{-\epsilon})-\epsilon$ satisfying $X_\epsilon(0)=X_{-\epsilon}(0)=X_0$. Then, for all $t>0$, \begin{equation} X_n(t)\in [X_\epsilon(t), X_{-\epsilon}(t)].\label{EncadrementXn}
\end{equation}

 It remains to show $\lim_{\epsilon\to 0^+}X_\epsilon(\omega)=\lim_{\epsilon\to 0^+}X_{-\epsilon}(\omega)=X(\omega)$.
We set $Z=X-X_\epsilon$. Since $h$ is Lipschitz continuous, there exists $C_h>0$ such that, for all $t\geq 0$, $$|Z'(t)|=|h(X(t))-h(X_\epsilon(t))-\epsilon|\leq C_h|X(t)-X_\epsilon(t)|+\epsilon=|Z(t)|+\epsilon.$$
Using the Gronwall lemma, 
 one gets $$\forall t\geq 0,\ |Z(t)| \leq |Z(0)|e^{C_ht}+\frac{\epsilon}{C_h}(e^{C_ht}-1)=\frac{\epsilon}{C_h}(e^{C_ht}-1).$$
Consequently $\lim_{\epsilon\to 0^+}Z(\omega)=0$ and so $\lim_{\epsilon\to 0^+}X_\epsilon(\omega)=X(\omega)$. Similarly, we can show $\lim_{\epsilon\to 0^+}X_{-\epsilon}(\omega)=X(\omega)$. By using the squeeze theorem and (\ref{EncadrementXn}), it follows that $\lim_{\epsilon\to 0^+}X_n(\omega)=X(\omega)$. 
\end{proof}

Using Lemma \ref{ConvergenceFlot} and Proposition \ref{PropConvexitéPhin}, we prove Theorem \ref{TheoremConvexite}.
\begin{proof}[Proof of Theorem \ref{TheoremConvexite}.]
Here we prove the convexity case. Proof for the concavity one is similar.
Let $x \in (a,b)$ and $(x_n)_{n\in \mathbb{N}}$ be a sequence such that $\forall n\in \mathbb{N},\ x_n\in (a,b)$ and $\lim_{n\to +\infty } x_n =x$. We first show that  \begin{equation}
 \lim_{n\to +\infty}h_n^*(x_n)=\lim_{n \to \infty}\sideset{^{*}}{_n}{\mathop{h}}(x_n)=h'(x), \label{Limitehnxn}
 \end{equation}
 where $h_n^*$ and $\sideset{^{*}}{_n}{\mathop{h}}$ are given by (\ref{Valeurshn'}).

 For all $n\in \mathbb{N}^*$, there exists a unique $i \in \{ 0,\dots, n-1 \}$ such that $x_n\in [c_{i}^{(n)},c_{i+1}^{(n)})$, we note $i=i_n$. Since $\lim_{n\to +\infty } x_n =x$ and $\forall i\in \{ 0,\dots n-1\}$, $$\lim_{n\to +\infty } |c_{i+1}^{(n)}-c_i^{(n)}|=\lim_{n\to +\infty } 1/n =0,$$ it follows that \begin{equation}\lim_{n\to +\infty } c_{i_n-1}^{(n)}=\lim_{n\to +\infty } c_{i_n+1}^{(n)}=x. \label{Limitecn}\end{equation}

Using the mean value theorem and the definition of $a_i$ in (\ref{Valeurs_ai_bi_ci}), it follows that for all $n\in \mathbb{N}^*$, there exists $\alpha_n \in (c_{i_n-1}^{(n)},c_{i_n}^{(n)})$ and $\beta_n \in (c_{i_n}^{(n)},c_{i_n+1}^{(n)}) $ such that \begin{equation} a_{i_n}=h'(\alpha_n) \text{ and } a_{i_n+1}= h'(\beta_n).\label{TAF}\end{equation}
Since $h\in\mathcal{C}^1([a,b],\mathbb{R})$, using  (\ref{Limitecn}) and (\ref{TAF}), it follows that
\begin{equation}
\lim_{n\to +\infty}a_{i_n}=\lim_{n\to +\infty}a_{i_n+1}=h'(x). 
\label{limiteai}
\end{equation}

  On the other hand, according to (\ref{Valeurshn'}), for all $n\in \mathbb{N}^*$, \begin{equation}
 h_n^*(x_n), \sideset{^{*}}{_n}{\mathop{h}}(x_n)\in \{a_{i_n},a_{i_n+1} \}.
\label{encradrementhn'} \end{equation}
With (\ref{limiteai}) and (\ref{encradrementhn'}) we conclude that (\ref{Limitehnxn}) is true. In particular, 
 \begin{equation} \left\{ \begin{array}{l}
 \lim_{n\to +\infty}h_n^*(x)=\lim_{n \to \infty}\sideset{^{*}}{_n}{\mathop{h}}(x)=h'(x),\\
 \lim_{n\to +\infty}h_n^*(\varphi_{\omega}^n(x))=\lim_{n \to \infty}\sideset{^{*}}{_n}{\mathop{h}}(\varphi_{\omega}^n(x))=h'(\varphi(x)).
 \end{array}\right. \label{Convergence_hn'}
 \end{equation}

Let $X_0\in (a,b)$. If we assume $h'(\varphi_\omega(X_0))>h'(X_0)$, since $h'$ and $\varphi_\omega$ are continuous there exists a neighbourhood $I_{X_0}$ of $X_0$ such that, for all $y\in I_{X_0}$, $h'(\varphi_\omega(y))>h'(x)$. 

Moreover, (\ref{Convergence_hn'}) implies that there exists a neighbourhood $J_{X_0}\subset I_{X_0}$ of $X_0$ and $N_0\in \mathbb{N}^*$ such that for all $n\geq N_0$ and for all $y\in J_{X_0}$, \begin{equation}
\sideset{^{*}}{_n}{\mathop{h}}(y)<\sideset{^{*}}{_n}{\mathop{h}}(\varphi_{\omega}^n(y)) \text{ and } h_n^*(y)<h_n^*(\varphi_{\omega}^n(y)). \nonumber
 \end{equation}
According to Proposition \ref{PropConvexitéPhin}, for all $n\geq N_0$, $\varphi_{\omega}^n$ is strictly convex on $J_{X_0}$.

Finally, as the limit of a sequence of strictly convex functions (Lemma \ref{LemmeConvexité}), $\varphi_{\omega}$ is convex on $J_{X_0}$.  
 \end{proof}

In the following section, we apply results of Sections \ref{PropGénérales}, \ref{SectionSolPeriodiques} and \ref{sectionConvex} to a special case that motivated that study.

\section{Equation on Tbet with impulses}\label{ChapitreImpulseTbet}
\subsection{Introduction}\label{introTbet}
In the multiscale model presented in \citet{Prokopiou} and \citet{Gao2016}, protein Tbet drives the development of effector properties and cell death in the CD8 T-cell population. In this model, all effector cells develop the same phenotype and the differentiation into memory cells is not considered. 

We are interested in studying how unequal partitioning of molecular content at cell division can explain the emergence of two subpopulations of effector CD8 T-cells, with the first ones \--- characterised by a low level of Tbet \--- able to develop memory properties and the second ones  \--- characterised by a high level of Tbet  \--- destined to die during the contraction phase \citep{Kaech2012,Lazarevic2013}. We hypothesize  that when a CD8 T-cell divides, the intracellular content of the mother cell is not split in two equal parts but asymmetric division occurs and gives birth to two daughter cells with different molecular profiles \citep{Sennerstam88,Block1990} (here, different concentrations of Tbet). Each daughter cell is supposed to be two times smaller than the mother cell so that the mass is preserved.

  To model this phenomenon, we consider the IDE with fixed impulse times (\ref{ImpulsiveEquationFeedback}),  which allows us to describe the effect of division, occurring at discrete times $\tau_k$, $k\in \mathbb{N}^*$, on the concentration $X$ of protein Tbet in a single CD8 T-cell throughout several cell divisions,
\begin{equation}
{\renewcommand{\arraystretch}{2}
\left \{ 
\begin{array}{r l l}
\dfrac{\mathbf{d}X(t)}{\mathbf{d}t} =&\eta\dfrac{X(t)^n}{\theta^n+X(t)^n}-\delta X(t),&\ t\in \mathbb{R}^+\backslash\{ \tau_k, k\in \mathbb{N}^* \}, \\
X(\tau_k^+)=&(1+\alpha_k)X(\tau_k^-),&\ k\in \mathbb{N}^*, \\
X(0)=&X_0\in \mathbb{R}^+.
\end{array}
\right. \\ 
\label{ImpulsiveEquationFeedback}}
\end{equation}
 At any time $\tau_k,\ k\in\mathbb{N}^*$, a cell divides and transmits a fraction $(1+\alpha_k)$ of its concentration of Tbet to one daughter cell  (the other one, which is not modelled, receives a concentration $(1-\alpha_k)X(\tau_k^-)$, so that the mean concentration in the two daughter cells remains equal to $X(\tau_k^-)$) \citep{Prokopiou,Gao2016}.  
 

The dynamics of the concentration $X$ within a single cell between two consecutive divisions is modelled by

\begin{equation} 
\dfrac{\mathbf{d}X}{\mathbf{d}t} =f(X):=\eta\dfrac{X^n}{\theta^n+X^n}-\delta X, \label{EquationFeedback}  \\ 
\end{equation}
where $n,\ \eta,\ \theta$ and $\delta$ are positive constants.
 The first term on the right hand side of the equation accounts for a positive feedback loop between protein Tbet and its own gene \textit{Tbx21} \citep{Kanhere2012,Prokopiou,Gao2016}, and parameter $\delta$ accounts for protein decay as well as  dilution due to cell growth. 

A common approach to study an impulsive system such as (\ref{ImpulsiveEquationFeedback}) is to turn it into a differential equation without impulse and study the latter \citep{Yan1998}. Indeed, let us introduce the non-impulsive system (\ref{SystReduit}), obtained from (\ref{ImpulsiveEquationFeedback}),
\begin{equation}
{\renewcommand{\arraystretch}{2}
\left \{ 
\begin{array}{r l l}
\dfrac{\mathbf{d}\mathscr{X}(t)}{\mathbf{d}t}=& P(t)\eta\dfrac{\mathscr{X}(t)^n}{P(t)^n\theta^n+\mathscr{X}(t)^n}-\delta\mathscr{X}(t),& t\in \mathbb{R}^+\backslash\{ \tau_k, k\in \mathbb{N}^* \}, \\
\mathscr{X}(0)=&X_0,
\end{array}
\right. 
\label{SystReduit}}
\end{equation}
where $P$ is the discontinuous function defined by $P(t)=  {\prod}_{{0<\tau_k\leq t}}(1+\alpha_k)^{-1} $. 
 The following proposition is inspired by Theorem 1 from \citet{Yan1998}, the proof is analogous  and will be omitted.

\begin{proposition}
Let  $\mathbf{(H_1)}$ and $\mathbf{(H_2)}$ hold true. Then $X$ is solution of (\ref{ImpulsiveEquationFeedback}) if and only if $\mathscr{X}: t\mapsto P(t)X(t)$ is solution of (\ref{SystReduit}).\label{PropReduction}
\end{proposition}

\begin{remark} Proposition \ref{PropReduction} is particularly useful under the strong assumption that function $P$ is periodic on $\mathbb{R}^+$  \citep{Yan2003,Yan2005,Li2005,Saker2007,Kou2009}. Indeed, in that case, the periodicity of $\mathscr{X}$ implies the periodicity of $X$ and then  existence and stability of periodic solutions for System (\ref{ImpulsiveEquationFeedback}) can be studied with the non-impulsive system (\ref{SystReduit}), generating smooth solutions. Some authors succeeded in obtaining results on the existence and stability of periodic solutions with weaker hypotheses \citep{Liu2007,Faria2016}, however for specific right-hand sides of the equation. 
As mentioned by \citet{Liu2007}, the hypothesis of periodicity of the function $P$ is too restrictive. Moreover it has no simple physical interpretation and is sometimes abusively used, leading to unrealistic conclusions.  
It would not be relevant for the biological process we are studying to consider that $P$ is periodic, therefore we do not present results obtained under this hypothesis. 
\end{remark}

We want to apply the results of Section    \ref{sectionConvex} to System (\ref{ImpulsiveEquationFeedback}). To this end, we first focus on the convexity of $f$, defined by (\ref{EquationFeedback}).
\begin{proposition}
Assume $n>1$ and $\Theta=\theta\sqrt[n]{\frac{n-1}{n+1}}$. Then  $f$ is strictly convex on $[0,\Theta)$, strictly concave on $(\Theta,+\infty)$ and has an inflexion point in $\Theta$.   \label{LemmeInflexion}
\end{proposition}
The proof of Proposition \ref{LemmeInflexion} is straightforward.

Since $f$ is Lispschitz continuous on $\mathbb{R}^+$ and $f(0)=0$, System (\ref{ImpulsiveEquationFeedback}) is a particular case of System (\ref{ImpulsiveEquation}) where $g=f$ and $U=\mathbb{R}^+$. Consequently the results presented in Section \ref{PropGénérales} hold true for System (\ref{ImpulsiveEquationFeedback}).

Since we only focus on nonnegative solutions, we only consider nonnegative steady states of the autonomous equation (\ref{EquationFeedback}), without impulse. Their existence and stability are given by Proposition \ref{équilibreTbet}, whose proof is omitted.

\begin{proposition}Let $n>1$, then   \begin{description}
\item[(i)] If $\eta(n-1)^{\frac{n-1}{n}}<n\delta \theta$, the trivial solution $X\equiv 0$ is the unique steady state of (\ref{EquationFeedback}) and it is globally asymptotically stable (on $\mathbb{R}_+$).
\item[(ii)]  If $\eta(n-1)^{\frac{n-1}{n}}=n\delta \theta$, equation (\ref{EquationFeedback}) has exactly two steady states: $X\equiv 0$, which is locally asymptotically stable, and $X\equiv \theta\sqrt[n]{n-1}$ which is unstable (attractive on $\left[ \theta\sqrt[n]{n-1},+\infty\right)$ but repulsive on $\left( 0,\theta\sqrt[n]{n-1}\ \right)$).
\item[(iii)] If $\eta(n-1)^{\frac{n-1}{n}}>n\delta \theta$,  (\ref{EquationFeedback}) has exactly three steady states: $0<X_u<X_s$, such that $X\equiv 0$ is locally asymptotically stable, $X_u$ is unstable and $X_s$ is locally asymptotically stable (with basin of attraction $(X_u,+\infty)$). Moreover $0<X_u<\theta \sqrt[n]{n-1}<X_s<\eta/\delta$. 
\end{description}
 \label{équilibreTbet}\end{proposition}
 The three statements of Proposition \ref{équilibreTbet} are illustrated on Figure \ref{PortraitPhaseTbFB}.
%
%
\begin{figure}[t]
\begin{center}
\includegraphics[scale=0.68,  angle =0 ]{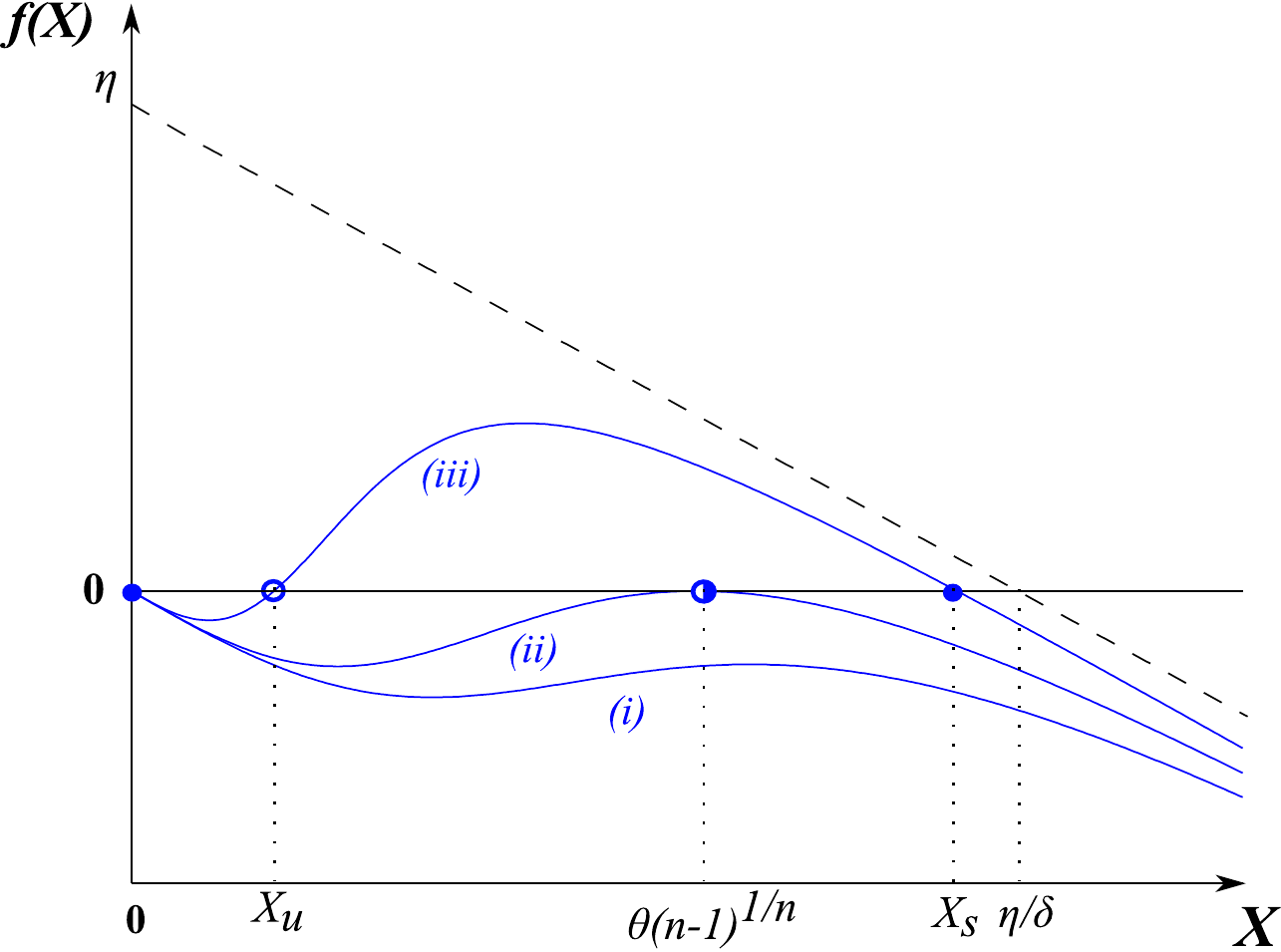}
	\caption{Steady states of  (\ref{EquationFeedback}). Straight blue curves correspond to the different cases given in Proposition \ref{équilibreTbet}. Dotted line is the graph of the function $X \mapsto\eta-\delta X$.}
	\label{PortraitPhaseTbFB}
\end{center}
\end{figure}
For biological reasons, in the following we will focus on the case when System (\ref{EquationFeedback}) is bistable, i.e. we will assume that the following hypothesis holds true:  \begin{description}
\item[($\mathbf{H_5}$)]$n>1$ and $\eta(n-1)^{\frac{n-1}{n}}>n\delta\theta$.
\end{description} 

Note that under ($\mathbf{H_5}$), $\Theta:=\theta\sqrt[n]{(n-1)/(n+1)}<\theta \sqrt[n]{n-1}$. Therefore, according to Proposition \ref{équilibreTbet}-\textit{(iii)}, $\Theta\in (0,X_s)$.


\subsection{Existence of periodic solutions for (\ref{ImpulsiveEquationFeedback})}\label{EtudeImpulsiveTbet}
Throughout this section, we suppose that hypotheses $\mathbf{(H_3)}$, $\mathbf{(H_4)}$ and $\mathbf{(H_5)}$ hold true and we investigate the existence of periodic solutions of (\ref{ImpulsiveEquationFeedback}). We have seen in Proposition \ref{periode} that periodic solutions are either constant, or their smallest period is $\omega$.

From now on, $\varphi_\omega$ stands for the flow of System (\ref{ImpulsiveEquationFeedback}) at time $\omega$, as defined in (\ref{defVarphi}). Here $f(0)=0$ so, according to Remark \ref{ContinuitéGamma},  $\varphi_\omega^{-1}(0)=0$ and the function $\Gamma_\omega$, defined by (\ref{DefGamma}), is continuous on $\mathbb{R^+_*}$. Moreover $X(\ \cdot\ ; X_0, \alpha$) is periodic if and only if $X_0=0$  or $X_0\in \Gamma_\omega ^{-1} (\{\alpha\})$ (Proposition \ref{PropGamma}). In the rest of this section, we show the following theorem.
 \begin{theorem} Let hypotheses $\mathbf{(H_3)}$, $\mathbf{(H_4)}$ and $\mathbf{(H_5)}$ hold true. Then, there exists $\alpha^*\in(-1,0)$ and a non-empty interval $I=[I_m,I_M]\subset [X_u,X_s)$ such that \begin{itemize}
\item If $\alpha < \alpha^*$,  $X\equiv 0$ is the only periodic solution of (\ref{ImpulsiveEquationFeedback}) and for any $X_0>0$, $\lim_{t\to + \infty}X(t;X_0,\alpha)=0$.
\item 
$X(\ \cdot \ ,X_0,\alpha^*)$ is periodic if and only if $X_0\in \{0\}\cup I$.
\item If $\alpha^* < \alpha< \exp(\delta \omega)-1$, there are exactly 3 periodic solutions of (\ref{ImpulsiveEquationFeedback}): $X\equiv 0$, $X(\ \cdot \ ;X_l,\alpha)$ and $X(\ \cdot \ ;X_h,\alpha)$ where $\{X_l,X_h\}=\Gamma_\omega^{-1}(\{\alpha\}) $ such that $0<X_l<X_h$ and these solutions are respectively asymptotically stable and attractive on $[0,X_l)$, unstable and repulsive on $(0,X_l)\cup (X_l,X_h)$ and asymptotically stable and attractive on $(X_l,+\infty)$. Moreover, if $\alpha<0$, $X_l\in (0,X_u)$ and $X_h\in (X_s, +\infty)$, if $\alpha=0$, $X_l=X_u$ and $X_h=X_s$, if  $\alpha>0$, $X_l\in (X_u,I_m)$ and $X_h\in (I_M,X_s)$.
\item If $\alpha \geq \exp(\delta\omega)-1$, $X\equiv 0$ is the only periodic solution of (\ref{ImpulsiveEquationFeedback}). Moreover, for any $X_0>0$, $\lim_{t\to + \infty}X(t;X_0,\alpha)=+\infty$.
\end{itemize}
In addition, if $\Theta\neq X_u$, $I\subset (X_u,X_s)$ and if $\Theta=X_u$, $I=\{X_u\}$. \label{TheoremeSolPeriodiques}
\end{theorem}
To prove Theorem \ref{TheoremeSolPeriodiques}, we need to apply Proposition \ref{PropGamma} and Lemma \ref{monotonieselongamma} from Section \ref{SectionSolPeriodiques}. To this end, we first introduce  intermediate results on the behaviour of $\varphi_{\omega}$ and $\Gamma_\omega$.
\begin{lemma} Let hypothesis $\mathbf{(H_5)}$ hold true. Then, for any $\omega>0$, 
\begin{equation}
\underset{X_0\to +\infty}{\lim}\Gamma_\omega (X_0)=\underset{X_0\to 0^+}{\lim}\Gamma_\omega (X_0)=\exp (\delta\omega)-1. \nonumber
\end{equation}\label{lemmeLimites}
\end{lemma}
\begin{proof}
We first rewrite equation (\ref{EquationFeedback}) in the form $X'(t)=\mathcal{H}\left( X(t)\right) -\delta X(t)$
  with   \begin{equation}
\mathcal{H}:  x\in \mathbb{R}^+ \mapsto  \eta\dfrac{x^n}{\theta^n+x^n}\in \left[ 0,\eta\right).
\nonumber
\end{equation}
Hence, for all $ X_0 \in \mathbb{R^+_*}$, $$\varphi_\omega(X_0)=X(\omega;X_0,0)=X_0e^{-\delta\omega}+e^{-\delta\omega}\int_0^\omega e^{\delta s}\mathcal{H}\left( X(s)\right) \mathrm{d} s,$$
thus \begin{equation}\frac{X(\omega;X_0,0)}{X_0}=e^{-\delta\omega}+e^{-\delta\omega}\int_0^\omega e^{\delta s}\frac{\mathcal{H}\left( X(s)\right)}{X_0} \mathrm{d}s.\label{ExpressionVarphi}\end{equation}
Since $\mathcal{H}<\eta$ on $\mathbb{R}^+$ and $e^{\delta s}\leq e^{\delta \omega}$ for $s\in(0,\omega)$, one obtains $$0\leq e^{-\delta\omega}\int_0^\omega e^{\delta s}\frac{\mathcal{H}\left( X(s)\right)}{X_0} \mathrm{d}s \leq 
\omega \frac{\eta}{X_0} . $$
Then it is clear from (\ref{ExpressionVarphi}) that $$\underset{X_0\to + \infty}{\lim} \frac{X(\omega;X_0,0)}{X_0}=\exp (-\delta\omega)$$ and then, from (\ref{DefGamma}),
$\underset{X_0\to + \infty}{\lim}\Gamma_\omega (X_0)=  \exp (\delta\omega)-1$.

We now study the limit of $\Gamma_\omega$ as $X_0$ goes to zero.
Let us recall that if $X_0<X_u$, then $f(X_0)<0$ and $X(\ \cdot \ ;X_0,0)$ decreases on $[0,\omega]$. Consequently, for $s\in [0,\omega]$, $X(\omega)\leq X(s) \leq X_0$ and we have the inequality $$\frac{\mathcal{H}\left( X(s)\right)}{X_0} = \eta\dfrac{X(s)^n}{\left( \theta^n+X(s)^n\right)X_0 } \leq \eta\dfrac{X_0^{n-1}}{ \theta^n+X(\omega)^n } . $$
Therefore $$e^{-\delta\omega}\int_0^\omega e^{\delta s}\frac{\mathcal{H}\left( X(s)\right)}{X_0} \mathrm{d}s \leq 
\omega \eta\dfrac{X_0^{n-1}}{ \theta^n+X(\omega)^n }  . $$ Finally, from (\ref{ExpressionVarphi}), $$\underset{X_0\to 0^+}{\lim}\frac{X(\omega;X_0,0)}{X_0}=\exp(-\delta \omega),$$ so $\underset{X_0\to 0^+}{\lim}\Gamma_\omega (X_0)=\exp (\delta\omega)-1$. This concludes the proof.  
\end{proof}

\begin{lemma}
Let hypothesis $\mathbf{(H_5)}$ hold true, $\omega>0$ and $\Theta$ as defined in Proposition~\ref{LemmeInflexion}.  There exists $X_c\in (0,X_s)$ such that $\varphi_\omega$ is convex on $(0,X_c)$ and concave on $(X_c,+\infty)$. Moreover, $X_c=X_u$ if and only if $ X_u=\Theta$. \label{LemmeConvexitéPhi}
\end{lemma}
\begin{proof}
Let us assume $\Theta\in (X_u,X_s)$. The proof is similar if $\Theta \in (0,X_u]$. In order to apply Theorem \ref{TheoremConvexite}, we investigate the sign of $f'-f'\circ \varphi_\omega$ on $(0,+\infty)$.

As steady states of (\ref{EquationFeedback}), $0,\ X_u$ and $X_s$ are fixed points of $\varphi_\omega$. Therefore, if $x\in \{0,X_u,X_s\}$, then $f'(x)-f'(\varphi_\omega(x))=0$.

If $X_0\in (0,X_u)\cup(X_s,+\infty)$, then $f(X_0)<0$ and the solution of (\ref{ImpulsiveEquationFeedback}) decreases on $[0,\omega)$, so $0<\varphi_\omega(X_0)<X_0$. According to Proposition \ref{LemmeInflexion}, $f$ is strictly convex on $(0,X_u)\subset(0,\Theta)$ and strictly concave on $(X_s,+\infty)\subset(\Theta,+\infty)$. Consequently, for all $X_0\in(0,X_u)$, $f'(X_0)-f'(\varphi_\omega(X_0))>0$  and for all $X_0\in (X_s,+\infty)$, $f'(X_0)-f'(\varphi_\omega(X_0))<0$. 

If $X_0\in (X_u,X_s)$, then $f(X_0)>0$, the solution of (\ref{ImpulsiveEquationFeedback}) increases on $[0,\omega)$ and $\varphi_\omega(X_0)>X_0$. Moreover,  if $X_0\in (X_u,\varphi_\omega^{-1}(\Theta))$, then $\varphi_\omega(X_0)\in (X_0,\Theta)$, that is $X_u<X_0<\varphi_\omega(X_0)<\Theta$. Since  $f$ is strictly convex on $(X_u,\Theta)$ (Proposition \ref{LemmeInflexion}), it follows $f'(X_0)-f'(\varphi_\omega(X_0))<0$
. Similarly, for all $X_0$ in $[\Theta,X_s)$, $f'(X_0)-f'(\varphi_\omega(X_0))>0$. 
It remains to prove that there exists a unique $X_c\in (\varphi_\omega^{-1}(\Theta),\Theta)$ such that $f'(X_c)-f'(\varphi_\omega(X_c))=0$.

Since $f, \varphi_\omega \in \mathcal{C}^{\infty}(\mathbb{R})$, it follows $f'-f'\circ \varphi_\omega \in \mathcal{C}^{\infty}(\mathbb{R})$  and then there exists $X_c \in (\varphi_\omega^{-1}(\Theta),\Theta)$ such that $f'(X_c)-f'(\varphi_\omega(X_c))=0$.  To show the uniqueness of $X_c$, we suppose there exists $Y_c\in (\varphi_\omega^{-1}(\Theta),\Theta)$, $Y_c>X_c$, such that $f'(Y_c)-f'(\varphi_\omega(Y_c))=0$. Since $\varphi_\omega$ is increasing and $\varphi_\omega(X_s)=X_s$, it follows $X_c<Y_c<\Theta<\varphi_\omega(X_c)< \varphi_\omega(Y_c)<\varphi_\omega(\Theta)<X_s$. From Proposition \ref{LemmeInflexion}, $f$ is convex on $(\varphi_\omega^{-1}(\Theta),\Theta)$ so $f'(X_c)<f'(Y_c)$.
On the other hand, from Proposition \ref{LemmeInflexion}, $f$ is concave on $(\Theta,X_s)$ so  $f'(\varphi_\omega(X_c))>f'(\varphi_\omega(Y_c))$. There is a contradiction since $f'(X_c)=f'(\varphi_\omega(X_c))$ and $f'(Y_c)=f'(\varphi_\omega(Y_c))$.

 Then $f'(\varphi_\omega(x))=f'(x)$ if and only if $x\in \{0,X_u,X_c,X_s\}$ and $f'(\varphi_\omega(x))-f'(x)<0$ (resp. $>0$) if $x\in (0,X_u)\cup (X_c,X_s)$ (resp.  if $x\in (X_u,X_c)\cup [X_s,+\infty)$). 

According to Theorem \ref{TheoremConvexite}, and Remark \ref{Cashnegative} since $f<0$ on $(0,X_u)\cup(X_s,+\infty)$, we conclude that $\varphi_\omega$ is convex on $(0,X_u)$ and on $(X_u,X_c)$ and concave on $(X_c,X_s)$ and on $(X_s,+\infty)$. Since $\varphi_\omega\in \mathcal{C}^\infty (\mathbb{R}^+)$, then $\varphi_\omega$ is convex on $(0,X_c)$ and concave on $(X_c,+\infty)$. 

In the particular case  $X_u=\Theta$, $f$ is strictly convex on $(0,X_u)$ and strictly concave on $(X_u,+\infty)$. Thus $f'(\varphi_\omega(x))-f'(x)<0$ if $x\in (0,X_u)\cup (X_u,X_s)$ and  $f'(\varphi_\omega(x))-f'(x)<0$ if $x\in (X_s,+\infty)$. It follows that $\varphi_\omega$ is convex on $(0,X_u)$ and concave on $(X_u,+\infty)$, i.e. $X_c=X_u$. The converse is straightforward.  
\end{proof}

\begin{lemma}
Let hypothesis $\mathbf{(H_5)}$ hold true and $\omega>0$. Then, there exists $\beta>1$ and $x_\beta>0$ such that for all $x>0$, $\varphi_\omega(x)\leq \beta x$ and $\varphi_\omega(x_\beta)=\beta x_\beta$. Moreover, if there exists $x>0$ such that $\varphi_\omega(x)=\beta x$, then $\varphi_\omega'(x)=\beta$.\label{LemmeBeta}
\end{lemma}
\begin{proof}
The function $x\mapsto \varphi_\omega(x)/x$ is continuous on $(0,+\infty)$ and, in particular, on the closed interval $[X_u,X_s]$. Therefore there exists $\beta=\max_{x\in[X_u,X_s]}\varphi_\omega(x)/x$. Moreover, for all $x\in (X_u,X_s)$, $\varphi_\omega(x)/x>1$ and for all $x\in (0,X_u]\cup [X_s,+\infty)$, $\varphi_\omega(x)/x\leq 1$  (cf. proof of Lemma \ref{LemmeConvexitéPhi}). Consequently $\beta=\max_{x>0}\varphi_\omega(x)/x$ and $\beta>1$. Hence there exists $x_\beta\in [X_u,X_s]$ such that $\varphi_\omega(x_\beta)=\beta x_\beta$ and for all $x>0$, $\varphi_\omega(x)\leq \beta x$.

Let $x>0$ be such that $\varphi_\omega(x)/x=\beta$. Then $x\in \arg\max_{x>0}(\varphi_\omega/id)$, where $id$ denotes the  identity function on $\mathbb{R}^+$. As an extremum of  a $\mathcal{C}^\infty$-function, $x$ satisfies $(\varphi_\omega/id)'(x)=0$. That is 
 $\varphi_\omega'(x)= \varphi_\omega(x)/x=\beta$.  
\end{proof}

\begin{definition} For any $\omega>0$, let $X_c$ denotes the inflexion point of $\varphi_\omega$ introduced in Lemma \ref{LemmeConvexitéPhi},  $\beta:=\max_{x>0}(\varphi_\omega(x)/x)$, whose existence is stated in Lemma \ref{LemmeBeta} and $I:=\{x>0,\ \varphi_\omega(x)=\beta x \}$. It may be noted that $X_c$,  $\beta$ and $I$ depend on the value of $\omega$. \label{defXcBetaI}\end{definition}

\begin{lemma}
Let hypothesis $\mathbf{(H_5)}$ hold true and $\omega>0$. Then $I=[I_m,I_M]$ with $X_c<I_m\leq I_M<X_s$.\label{LemmeI}
\end{lemma}
\begin{proof}
From Lemma \ref{LemmeBeta} we know that $I$ is non-empty and $I\subset (X_u,X_s)$. Moreover, $I=(\varphi_\omega/id)^{-1}(\{\beta\})$, where $id$ is the identity function and $\varphi_\omega/id$ is continuous on $(0,+\infty)$, so $I$ is a closed set.
Let's show that $I\subset(X_c,X_s)$.  

By contradiction, let us suppose that there exists $x_0\in I$ such that $X_u<x_0 \leq X_c$. 
According to Lemmas  \ref{LemmeConvexitéPhi} and \ref{LemmeBeta}, for all $0<x<x_0$, $\varphi_\omega'(x)\leq \varphi_\omega'(x_0)=\beta$.
Then the function $x\mapsto \varphi_\omega(x)-\beta x$ is decreasing on $(0,x_0)$. Consequently, for all $x\in (0,x_0)$, $\varphi_\omega(x)-\beta x \geq \varphi_\omega(x_0)-\beta x_0=0$, by definition of $x_0$, and in particular $\varphi_\omega(X_u)-\beta X_u\geq 0$.

On the other hand, since $X_u$ is a steady state of (\ref{EquationFeedback}), then $\varphi_\omega(X_u)=X_u$ and since $\beta>1$ (Lemma \ref{LemmeBeta}), then $\varphi_\omega(X_u)-\beta X_u=X_u(1-\beta)<0$. There  is a contradiction, so $I\subset(X_c,X_s)$.

It remains to show that $I$ is an interval. If $I=\{x_\beta\}$, given by Lemma \ref{LemmeBeta}, the proof is complete. Assume $x_0,y_0\in I$ with $x_0<y_0$ and let $z_0\in (x_0,y_0)$. 
 Since $x_0>X_c$ and $\varphi_\omega$ is concave on $(X_c,+\infty)$, it follows from Lemma \ref{LemmeBeta} and Definition \ref{defXcBetaI} that $$\beta=\varphi_\omega'(x_0)\geq \varphi_\omega'(z_0)\geq \varphi_\omega'(y_0)= \beta.$$
Therefore $\varphi_\omega'(z_0)=\beta$ for all $z_0\in (x_0,y_0)$, so $\varphi_\omega$ is linear. Since $\varphi_\omega(x_0)=\beta x_0$, we deduce $\varphi_\omega(z_0)=\beta z_0$, that is $z_0\in I$ and $I$ is an interval. 
\end{proof}

\begin{lemma}
Let hypothesis $\mathbf{(H_5)}$ hold true, $\omega>0$, $I_m$, $I_M$ as defined in Lemma \ref{LemmeI} and $x\geq 0$. If $x\in (0,I_m)$, then $\varphi_\omega'(x)>\varphi_\omega(x)/x$ and if $x\in (I_M,+\infty)$, then $\varphi_\omega'(x)<\varphi_\omega(x)/x$. Otherwise, $\varphi_\omega'(x)=\varphi_\omega(x)/x$.
\label{LemmeComparaison}\end{lemma} 
\begin{proof}
We first focus on the interval $(0,X_c)\subset(0,I_m)$. Since $\varphi_\omega$ is convex on $(0,X_c)$ (Lemma \ref{LemmeConvexitéPhi}) and $\varphi_\omega(0)=0$, from the mean value theorem one has, for all $x\in(0,X_c)$, $\varphi_\omega'(x)\geq \varphi_\omega(x)/x$.

Let us suppose there exists $x_0\in (0,X_c)$ such that $\varphi_\omega'(x_0) =\varphi_\omega(x_0)/x_0$. 
Then we claim that for all $x\in (0,x_0)$, $\varphi_\omega(x)/x=\varphi_\omega(x_0)/x_0$. Indeed, from the convexity of $\varphi_\omega$, \begin{equation}\text{ for all }x\in (0,x_0),\ \varphi_\omega(x)/x\leq \varphi_\omega'(x)\leq \varphi_\omega'(x_0)=\varphi_\omega(x_0)/x_0.\label{RelationConvexité}\end{equation}
Assume there exists $x\in (0,x_0)$ such that $\varphi_\omega(x)/x<\varphi_\omega(x_0)/x_0$.  Then,  from the mean value theorem, there exists $c\in (x,x_0)$ such that \begin{equation}
\varphi_\omega'(c)=\frac{\varphi_\omega(x_0)-\varphi_\omega(x)}{x_0-x}>\frac{\varphi_\omega(x_0)-x\varphi_\omega(x_0)/x_0}{x_0-x}=\frac{\varphi_\omega(x_0)}{x_0}. \nonumber
\end{equation}
There is a contradiction with (\ref{RelationConvexité}), so for all $x\in (0,x_0)$, $\varphi_\omega(x)/x=\varphi_\omega(x_0)/x_0$.

On the other hand, according to Lemma \ref{lemmeLimites}, $$\lim_{x\to 0^+}\varphi_\omega(x)/x=\lim_{x\to 0^+}(\Gamma_\omega(x)+1)^{-1}=\exp(-\delta \omega).$$ It follows that $\varphi_\omega(x_0)/x_0=\exp(-\delta \omega)$ and so,  for all $x\in (0,x_0)$, $ \varphi_\omega(x)=\exp(-\delta \omega)x$. Consequently, $\varphi_\omega$ coincides with the flow of the linear equation $Y'(t)=-\delta Y(t)$. That is impossible since for all $x>0$,  $f(x)>-\delta x$, so the first assumption in the proof is false. We conclude that for all $x_0\in (0,X_c)$,  $\varphi_\omega'(x_0)> \varphi_\omega(x_0)/x_0$.

Now, consider $x_0\in [X_c,I_m)$ and suppose that $\varphi_\omega'(x_0)\leq  \varphi_\omega(x_0)/x_0$. By concavity of $\varphi_\omega$ on $(X_c,+\infty)$ (Lemma \ref{LemmeConvexitéPhi}), for all $x\geq x_0$, $\varphi_\omega'(x)\leq \varphi_\omega'(x_0)\leq  \varphi_\omega(x_0)/x_0$.
Consequently for any $y\in (x_0,+\infty),$ $$\varphi_\omega (y)=\varphi_\omega (x_0)+\int_{x_0}^y \varphi_\omega' (x) \mathrm{d}x \leq \varphi_\omega (x_0)+\frac{\varphi_\omega(x_0)}{x_0}(y-x_0)=y\frac{\varphi_\omega(x_0)}{x_0}.$$
In particular, $\varphi_\omega(I_m)\leq I_m \frac{\varphi_\omega(x_0)}{x_0}$. 
However, $\varphi_\omega(I_m)= I_m\beta$ (Lemma \ref{LemmeI}) and, since $x_0\notin I=[I_m,I_M]$ , $\beta> \varphi_\omega(x_0)/x_0$  (Lemma \ref{LemmeBeta}). There is a contradiction. Therefore, for all $x_0\in [X_c,I_m)$, $\varphi_\omega'(x_0)>  \varphi_\omega(x_0)/x_0$.

Similar arguments allow us to conclude that $\varphi_\omega'(x)<\varphi_\omega(x)/x$ for $x\in (I_M,+\infty)$.

Finally, if $x\in [I_m,I_M]$, from Lemmas \ref{LemmeBeta} and \ref{LemmeI}, $\varphi_\omega'(x)=\beta$ and $\varphi_\omega(x)=\beta x$. That achieves the proof. 
%
\end{proof}

Now, we can prove Theorem \ref{TheoremeSolPeriodiques}.
\begin{proof}[Proof of Theorem \ref{TheoremeSolPeriodiques}]
From the definition of $\Gamma_\omega$ in (\ref{DefGamma}), for $x>0$ one has $$\Gamma_\omega'(x)=\frac{\varphi_\omega(x)-x\varphi_\omega'(x)}{\varphi_\omega(x)^2}.$$
Then, from Lemma \ref{LemmeComparaison}, $\Gamma_\omega$ is decreasing on $(0,I_m)$, constant on $[I_m,I_M]$ with $\Gamma_\omega(x)=\alpha^*:=(1/\beta) -1$ (Lemma \ref{LemmeBeta}) and increasing on $(I_M,+\infty)$. Moreover from Lemma \ref{lemmeLimites},  $\lim_{x\to +\infty}\Gamma_\omega (x)=\lim_{x\to 0^+}\Gamma_\omega (x)=\exp (\delta\omega)-1$. In addition, since $X_u$ and $X_s$ are two steady states of (\ref{EquationFeedback}), $\Gamma_\omega(X_u)=\Gamma_\omega(X_s)=0$. So for the different values of $\alpha$ in Theorem \ref{TheoremeSolPeriodiques}, we can conclude on $\Gamma_\omega^{-1}(\{\alpha\})$ (see Figure \ref{FiguresGammaTheorique} (A)) and then on the number of periodic solutions of (\ref{ImpulsiveEquationFeedback}) (Proposition \ref{PropGamma}). 
The stability of the positive periodic solutions is a consequence of Lemma \ref{monotonieselongamma} and can be deduced from the graph of $\Gamma_\omega$ (Remark \ref{LireStabilité}).

Let's consider $\alpha^* \leq \alpha<\exp(\delta \omega)-1$, then $\Gamma_\omega^{-1}(\{\alpha\})=\{X_l,X_h\}$ with $X_l\in(0,I_m)$ and $X_h\in (I_M,+\infty)$. From the monotonicity of $\Gamma_\omega$ we conclude that $X(\ \cdot \ ; X_l,\alpha)$ is unstable and $X(\ \cdot \ ; X_h,\alpha)$ is asymptotically stable (Remark \ref{LireStabilité}).  It remains to study the stability of $X\equiv 0$. 
Let $X_0\in (0,X_l)$. Then, for all $t\geq 0$, $0<X(t;X_0,\alpha)<X(t;X_l,\alpha)$ and, as a periodic function, $X(\ \cdot \ ;X_l,\alpha)$ is bounded, so $X(\ \cdot \ ;X_0,\alpha)$ is bounded. According to Proposition \ref{borneeconverge}, $X(\ \cdot \ ;X_0,\alpha)$ converges to a periodic solution.
 Since $X(\ \cdot \ ;X_l,\alpha)$ is unstable, then $X(\ \cdot \ ;X_l,\alpha)$ converges necessarily to $0$, that is $X\equiv 0$ is asymptotically stable. 

Let $\alpha\geq\exp(\delta \omega)-1$, then $\Gamma_\omega^{-1}(\{\alpha\})=\{\emptyset\}$ and $X\equiv 0$ is the only periodic solution of (\ref{ImpulsiveEquationFeedback}). Moreover, let $X_0>0$, then $\Gamma_\omega(X_0)<\alpha$, that is $$\frac{X_0}{\varphi_\omega(X_0)}-1<\alpha \Leftrightarrow X(\omega;X_0,\alpha):=(1+\alpha)\varphi_\omega(X_0)>X_0. $$ 
Then, $(X(k\omega;X_0,\alpha))_{k\in \mathbb{N}^*}$ is an increasing sequence (Lemma \ref{monotonie}) and either converges to a periodic solution or is unbounded (Proposition \ref{borneeconverge}).  $X(\ \cdot \ ;X_0,\alpha)$ cannot converge to $0$ and therefore diverges to $+\infty$. 

Similarly, we show that if $\alpha < \alpha^* $, for any $X_0>0$, $(X(k\omega;X_0,\alpha))_{k\in \mathbb{N}^*}$ decreases. It follows that $0<X(\ \cdot \ ;X_0,\alpha)\leq  \max_{t\in [0,\omega)} X(t;X_0,\alpha)<+\infty$. 
 As a bounded solution of (\ref{ImpulsiveEquationFeedback}), $X(\ \cdot \ ;X_0,\alpha)$ converges to a periodic solution of (\ref{ImpulsiveEquationFeedback}) (Proposition \ref{borneeconverge}), and $X\equiv 0$ is the only one. 
  \end{proof}

\begin{figure}[t]
   \begin{minipage}[c]{.46\linewidth}
 \includegraphics[scale=0.4]{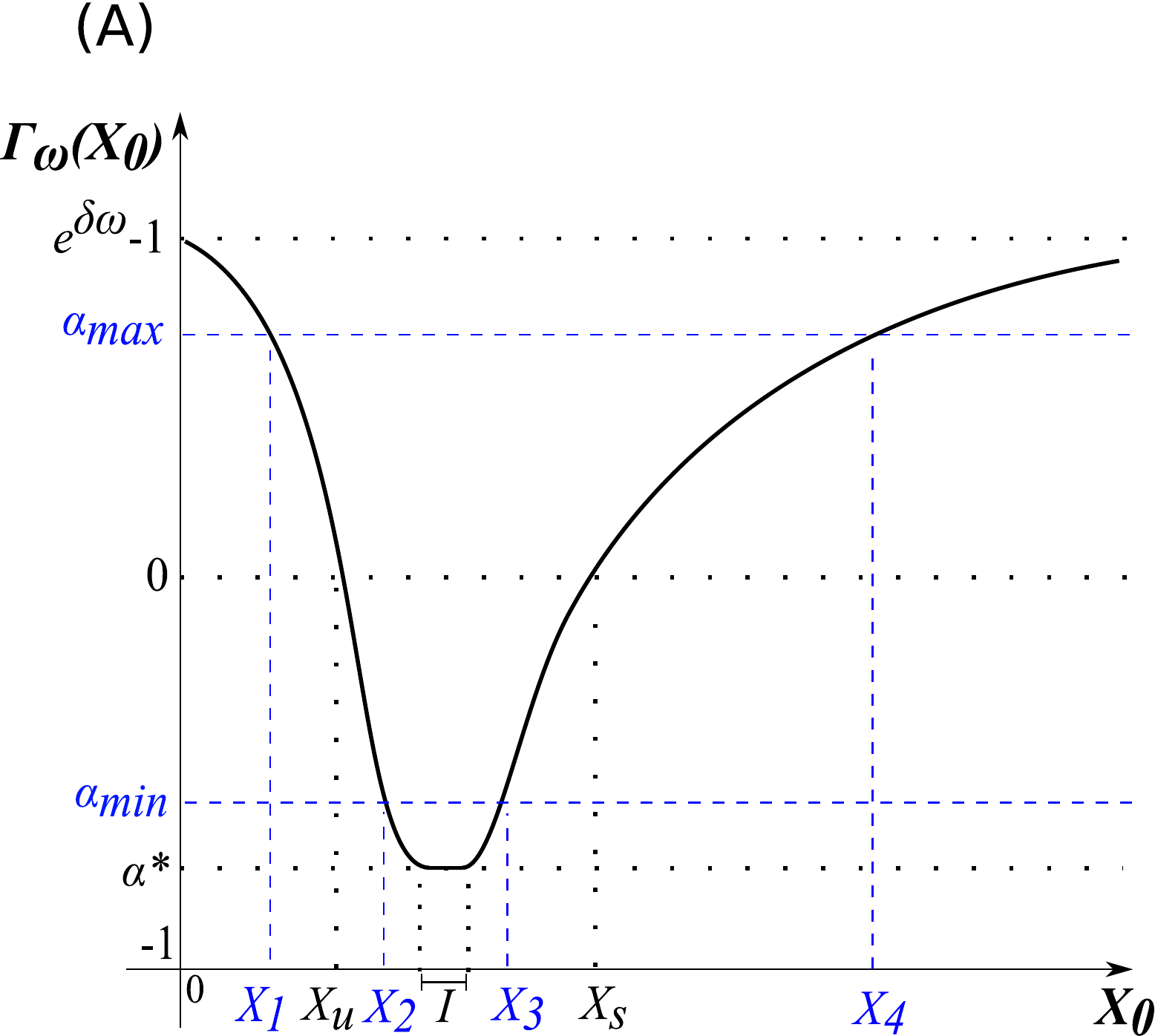}
   \end{minipage} \hfill
   \begin{minipage}[c]{.46\linewidth}
   \includegraphics[scale=0.4]{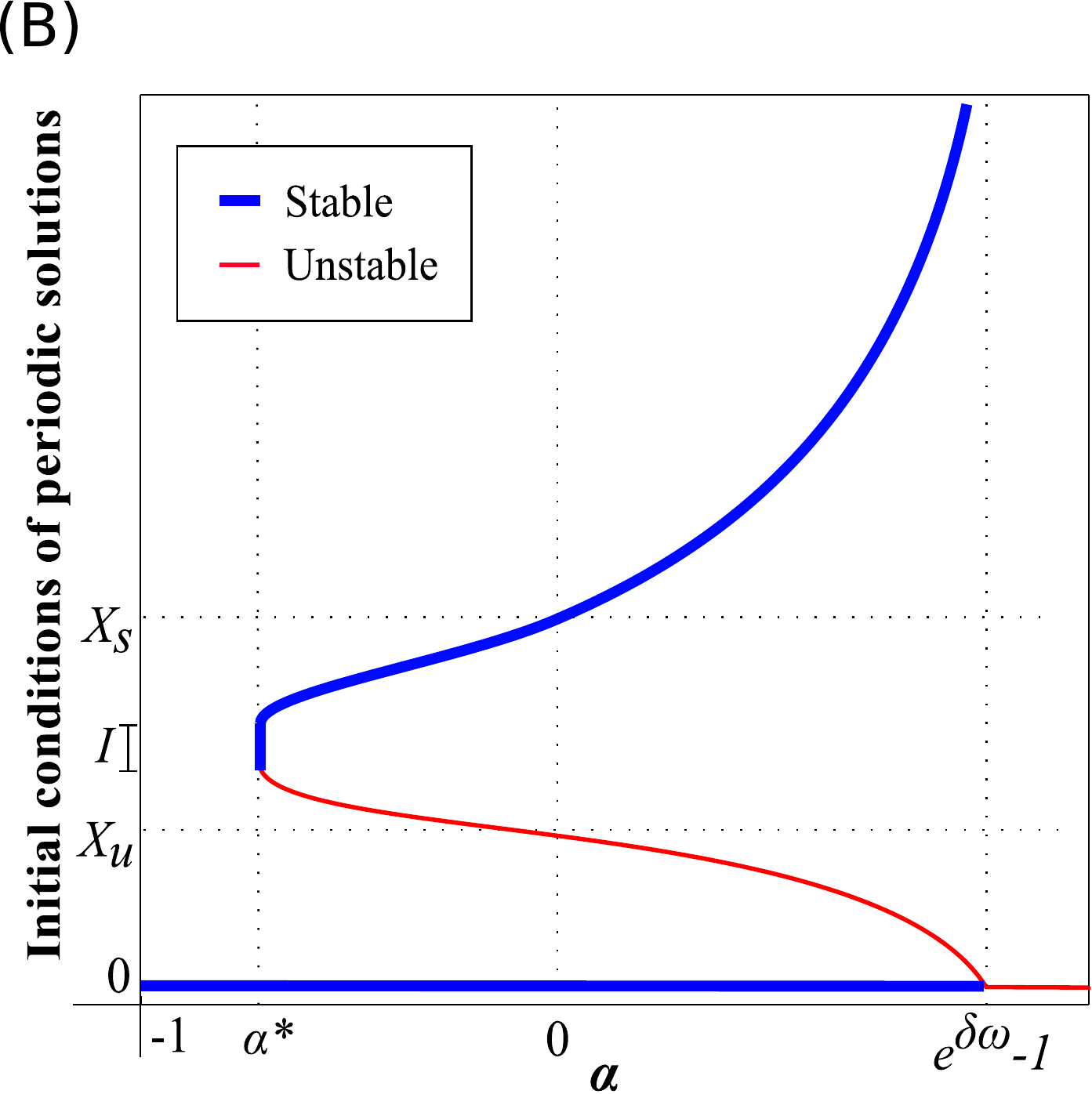}
   \end{minipage}\\
   \caption{According to the results shown in Section \ref{ChapitreImpulseTbet} and under hypotheses of Theorem \ref{TheoremeSolPeriodiques} : (A) Qualitative graph of function $\Gamma_\omega$. Blue dashed lines correspond to critical values mentioned in Proposition \ref{PropConclusion}. (B) Initial conditions of the periodic solutions of (\ref{ImpulsiveEquationFeedback}). For any $\alpha>-1$, the initial conditions $X_0$ such that $X(\ \cdot \ ;X_0,\alpha)$ is periodic and stable (resp. unstable) are drawn in bold blue (resp. thin red).}
   \label{FiguresGammaTheorique}
\end{figure}

\subsection{Numerical results}
\begin{figure}[t h]
\begin{center}
\includegraphics[scale=0.68,  angle =0 ]{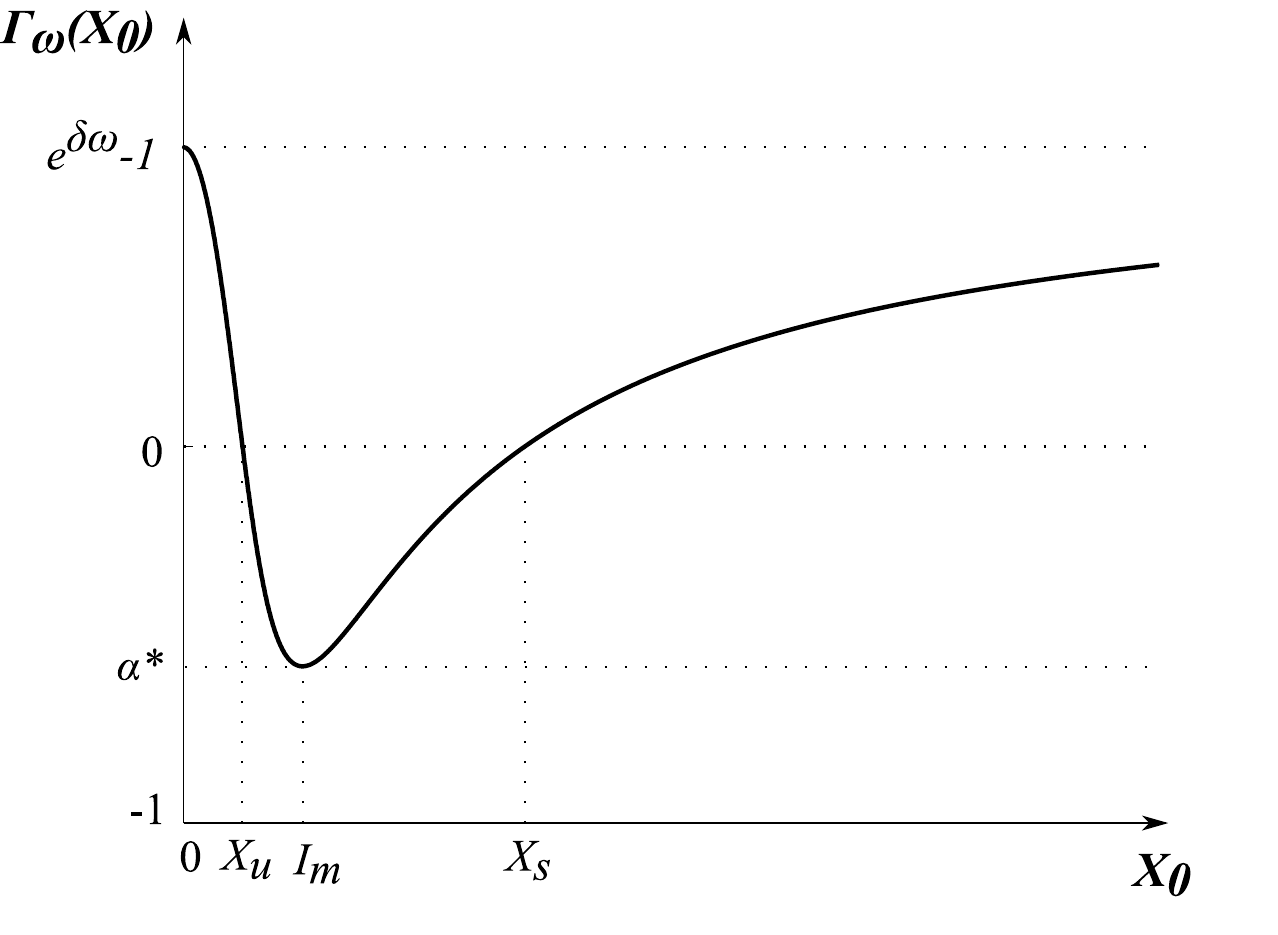}
\caption{Graph of function $\Gamma_\omega$ with parameters given by (\ref{param}), $X_u,\ X_s$ as defined in Proposition \ref{équilibreTbet}, $I_m$ from Lemma \ref{LemmeI} and $\alpha^*$ from Theorem \ref{TheoremeSolPeriodiques}.}
\label{graphGamma}
\end{center}
\end{figure}
Here, we numerically illustrate the results of Theorem~\ref{TheoremeSolPeriodiques} by setting $\omega=360\text{ min}$ and using the following set of parameters, satisfying  $\mathbf{(H_5)}$, \begin{equation} \begin{array}{l}
\eta=0.05 \text{ mol.L}^{-1}\text{ min}^{-1},\ \theta=35\text{ mol.L}^{-1}, \\ \delta=4.1 \times 10^{-4}  \text{ min}^{-1},\ n=3.
\end{array}
\label{param}\end{equation}
The graph of $\Gamma_\omega$, defined in (\ref{DefGamma}), with parameters (\ref{param}) is shown on Figure \ref{graphGamma}.
In particular, we observe that $I$, in Theorem \ref{TheoremeSolPeriodiques}, is reduced to $I=\{I_m\}$ with $I_m\approx 41.5\text{ mol.L}^{-1}$, $X_u\approx 20.5\text{ mol.L}^{-1}$, $X_s\approx 119\text{ mol.L}^{-1}$, $\alpha^*\approx -0.116$ and $\exp(\delta \omega)-1\approx 0.159$.
\begin{figure}[h! t]
\begin{center}
\includegraphics[scale=0.76,  angle =0 ]{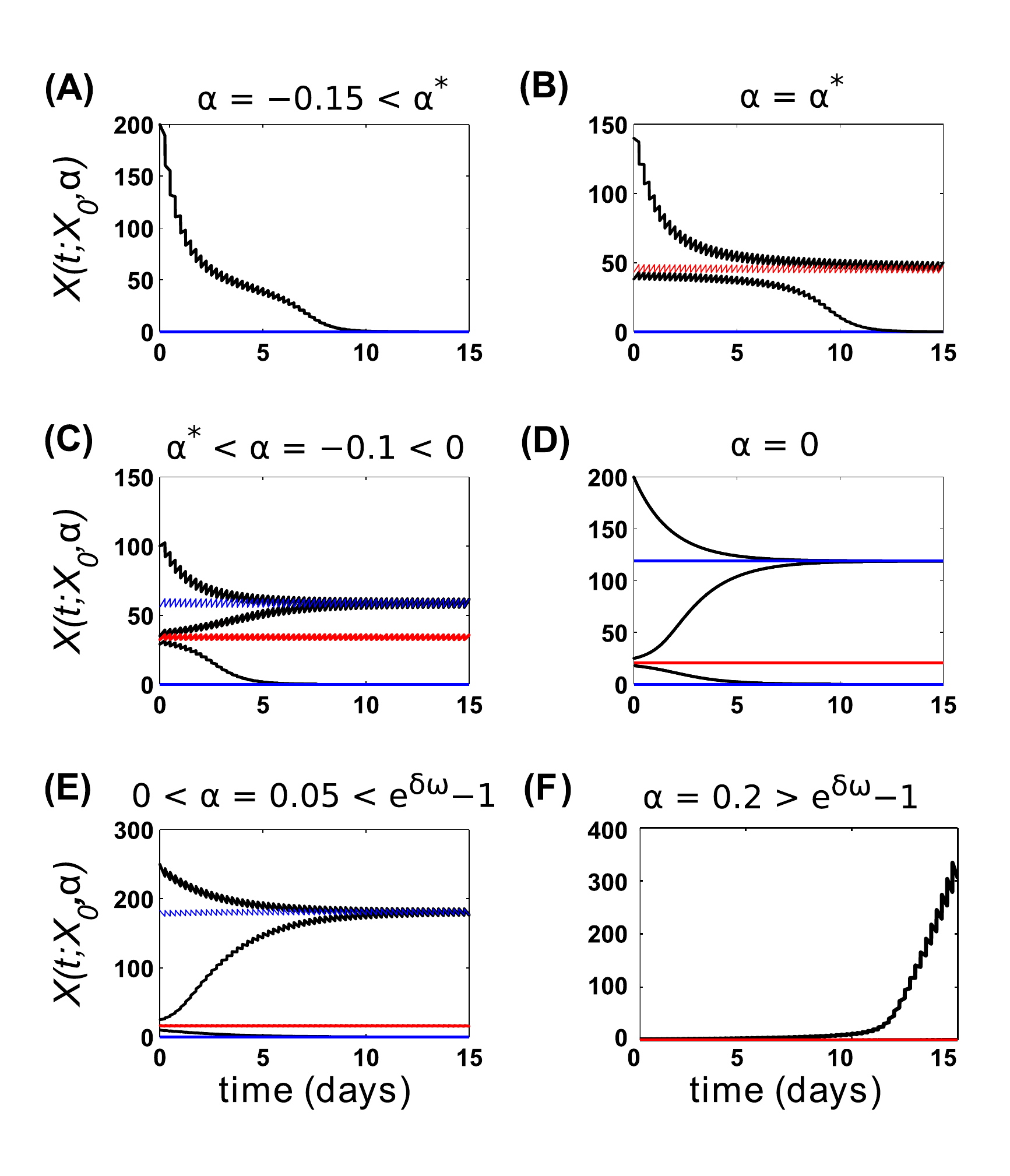}
\caption{Solutions of System (\ref{ImpulsiveEquationFeedback}) from different initial values and for different values of $\alpha$ corresponding to cases presented in Theorem \ref{TheoremeSolPeriodiques}. Stable (resp. unstable) periodic solutions are plotted in blue (resp. red).}
\label{FigSolPeriodiques}
\end{center}
\end{figure}

On Figure \ref{FigSolPeriodiques}, we illustrate  by numerical simulations the 6 qualitatively different cases mentioned in Theorem \ref{TheoremeSolPeriodiques}, that is: (A) $\alpha<\alpha^*$ and all positive solutions of (\ref{ImpulsiveEquationFeedback}) converge to the constant solution $X\equiv 0$; (B) $\alpha=\alpha^*$ and for any $X_0>I_m$ the solution $X(\ \cdot \ ;X_0,\alpha^*)$ converges to the periodic solution $X(\ \cdot \ ; I_m,\alpha^*)$ while for any $X_0<I_m$, the solution $X(\ \cdot \ ;X_0,\alpha^*)$ converges to the constant solution $X\equiv 0$; (C), (D), (E) are bistable cases: two stable periodic solutions are separated by an unstable one (in particular (D) corresponds to the case without impulse); (F) $\alpha>\exp(\delta\omega)-1$ and all positive solutions of (\ref{ImpulsiveEquationFeedback}) diverge to $+\infty$.

In this section, we established some qualitative properties of the function $\Gamma_\omega$, concluded on the existence and stability of periodic solutions of (\ref{ImpulsiveEquationFeedback}) under hypotheses $\mathbf{(H_3)}$, $\mathbf{(H_4)}$ and $\mathbf{(H_5)}$ and illustrated them with numerical simulations. In the following and last section, we replace $\mathbf{(H_4)}$ by a weaker hypothesis, allowing for stochastic partitioning of the molecular content at cell division, and draw conclusions on the biological problem of protein repartition at cell division.

\section{Application to cell fate decision}
\label{SectionBio}

In this section, we use results from Section \ref{ChapitreImpulseTbet} to draw conclusions on the behaviour of (\ref{ImpulsiveEquationFeedback}) when we no longer consider hypothesis $\mathbf{(H_4)}$ to be satisfied. Then, we propose an explanation, based on asymmetric partitioning of Tbet concentration at cell division, for the emergence of two qualitatively different pools of cells, with distinct fates, among a population generated by a single cell. 

As done in \citet{Prokopiou} and \citet{Gao2016}, we suppose that when a cell with concentration $X$ for protein Tbet undergoes its $k^{\text{th}}$ division, it gives birth to two daughters cells with concentrations $(1+\alpha_k)X$ and $(1-\alpha_k)X$ respectively. Biologically, it would not be so relevant to consider the  sequence $(\alpha_k)_{k\in \mathbb{N}^*}$ to be constant (hypothesis $\mathbf{(H_4)}$). Consequently, we introduce the weaker hypothesis $\mathbf{(H_6)}$,
\begin{description}
\item[($\mathbf{H_6}$)]  There exists $-1<\alpha_{min}<\alpha_{max}<+\infty$ such that the sequence $(\alpha_k)_{k\geq 1}$ verifies,  for all $k \in \mathbb{N}^*,\ \alpha_k \in [\alpha_{min},\alpha_{max}]\subset (-1,+\infty)$.
\end{description}

 According to Lemma \ref{encadrement}, if $\mathbf{(H_6)}$ holds true, then for any given initial condition $X_0$,  the solution $X(\ \cdot \ ;X_0,\alpha_k)$ of (\ref{ImpulsiveEquation}) is bounded by the two solutions $X(\ \cdot \ ;X_0,\alpha_{min})$ and $X(\ \cdot \ ;X_0,\alpha_{max})$ obtained with $m_k=\alpha_{min}$ and $M_k=\alpha_{max}$ respectively. Consequently, although ($\mathbf{H_4}$) is very restrictive on parameters $(\alpha_k)_{k\in\mathbb{N }^*}$, studying Systems (\ref{ImpulsiveEquation}) or (\ref{ImpulsiveEquationFeedback}) with ($\mathbf{H_4}$) provides bounds for the less restrictive case $\mathbf{(H_6)}$.

 The following proposition is then an immediate consequence of Lemma \ref{encadrement} and Theorem \ref{TheoremeSolPeriodiques}. Note that $X_1,\ X_2,\ X_3,\ X_4$ mentioned below are illustrated on Figure \ref{FiguresGammaTheorique}(A).

\begin{proposition}
Let hypotheses $\mathbf{(H_3)}$, $\mathbf{(H_5)}$ and $\mathbf{(H_6)}$ hold true. Let $\alpha^*$ from Theorem $\ref{TheoremeSolPeriodiques}$ and $\alpha_{min}$, $\alpha_{max}$ from $\mathbf{(H_6)}$ satisfy $\alpha^*<\alpha_{min}<0<\alpha_{max}<\exp (\delta \omega)-1$. Then there exists $X_1$, $X_2$, $X_3$, $X_4 \in \mathbb{R}^+$ such that 
 $\Gamma_\omega^{-1}(\{\alpha_{max}\})= \{X_1,X_4\},\ \Gamma_\omega^{-1}(\{\alpha_{min}\})=\{X_2,X_3\}$ and $$ 0<X_1<X_u<X_2<X_3<X_s<X_4. $$ 
 For the sake of simplicity, the periodic functions $X(\ \cdot \ ; X_1, \alpha_{max})$, $X(\ \cdot \ ; X_2, \alpha_{min})$,  $X(\ \cdot \ ; X_3, \alpha_{min})$ and $X(\ \cdot \ ; X_4, \alpha_{max})$ are denoted hereunder by $X_1^p,\ X_2^p,\ X_3^p$ and $X_4^p$  respectively. Then, for any sequence $(\alpha_k)_{k\in \mathbb{N}}$ satisfying $\mathbf{(H_6)}$, \begin{description}
 \item[i)] if $0\leq X_0<X_1$, $X(\ \cdot \ ; X_0, \alpha_k)$ converges to zero. 
 \item[ii)] if $X_1<X_0\leq X_2$,  for all $ t\geq 0,\ X(t ; X_0, \alpha_k)\in [0,X_4^p(t)]\subset [0,X_4]$. 
 \item[iii)] if $X_2<X_0<X_3$,   for all $ t\geq 0,\ X(t ; X_0, \alpha_k)\in (X_2^p(t),X_4^p(t))\subset (X_2,X_4)$. Moreover,  for any $\epsilon >0$ there exists $t^*>0$ such that, for all $ t>t^*$, $X(t ; X_0, \alpha_k) \in (X_3^p(t)-\epsilon,X_4^p(t)) \subset (X_3-\epsilon, X_4)$.
 \item[iv)]  if $X_3<X_0\leq X_4$,  for all $ t\geq 0,\ X(t; X_0, \alpha_k) \in (X_3^p(t),X_4^p(t))\subset (X_3, X_4)$. 
 \item[v)] if $X_4<X_0$,  for all $ t\geq 0,\ X(t ; X_0, \alpha_k)\in (X_3^p(t),X_0)\subset (X_3,X_0)$, moreover   for all $ \epsilon >0$ there exists $t^*>0$ such that  for all $ t>t^*,\ X(t ; X_0, \alpha_k)\in (X_3^p(t),X_4^p(t)+\epsilon) \subset (X_3, X_4+\epsilon)$.
 \end{description} \label{PropConclusion}
\end{proposition}
\begin{remark} \label{remarkq}
For straightforward biological reasons, $\alpha_k$ must be chosen in $(-1,1)$ 
  so the protein concentration in the daughter cell is positive and at most twice as large as that observed in the mother cell. From now on we consider that for all $k\geq 1$, $\alpha_k \in [q-1,1-q]$ where $q\in (0,1)$. For example, $\alpha_k$ can be randomly chosen from the uniform law $\mathcal{U}_{[q-1,1-q]}$ \citep{Prokopiou,Gao2016}.
\end{remark}

Hereinbelow, the concentration of protein Tbet in a cell is associated to its level of differentiation. High concentration of Tbet ($X\approx X_s$) corresponds to an effector phenotype while low concentration ($X\approx 0$) corresponds to a memory phenotype \citep{Joshi2007,Kaech2012,Lazarevic2013}. In the following, we say that the fate of a cell is irreversible if Tbet concentration in that cell, modelled by System (\ref{ImpulsiveEquationFeedback}), is definitively higher or definitively lower than $X_u$. Using the results from Theorem \ref{TheoremeSolPeriodiques} and Proposition \ref{PropConclusion}, we discuss how the value of $q$, that is, the degree of asymmetry in the process of protein distribution (see Remark \ref{remarkq}), impacts the cell population generated by a single cell and cell fate reversibility. Note that the smaller $q$, the more asymmetric the distribution.

Let $\mathbf{(H_3)}$, $\mathbf{(H_5)}$ and $\mathbf{(H_6)}$  hold true with $[\alpha_{min},\alpha_{max}]= [q-1,1-q]$, $q\in (0,1)$. According to Theorem \ref{TheoremeSolPeriodiques}, the solutions of (\ref{ImpulsiveEquationFeedback}) are bounded if and only if $\alpha_{max}<\exp (\delta \omega)-1$, that is $q>2-\exp (\delta \omega)$. Indeed, if it is not the case, there exists $X_0>0$ and a sequence $(\alpha_k)_{k\in \mathbb{N}}$  satisfying $\mathbf{(H_6)}$  such that $X(\ \cdot \ ; X_0, \alpha_k)$ tends to $+\infty$ (for example, $\alpha_k\equiv \alpha_{max}$). Then, it is reasonable to consider  $q>2-\exp (\delta \omega)$. Similarly, if $\alpha_{min}<\alpha^*$ (that is $q>1+\alpha^*$), for any $X_0>0$ there exists a sequence $(\alpha_k)_{k\in \mathbb{N}}$  satisfying $\mathbf{(H_6)}$  such that $X(\ \cdot \ ; X_0, \alpha_k)$ tends to $0$ (for example  $\alpha_k\equiv \alpha_{min}$). 
 
 
On the contrary, if we assume $1>q>\max (1+\alpha^*,2-\exp (\delta \omega))$, then $[q-1, 1-q]\subset (\alpha^*,\exp (\delta \omega)-1)$ and Proposition \ref{PropConclusion} holds.  In that case, if there exists $K\in \mathbb{N}$ such that $X(K\omega)<X_1$ (resp. $X(K\omega)>X_2$), then for any $(\alpha_k)_{k\geq K}$  satisfying $\mathbf{(H_6)}$, the concentration  $X(\ \cdot \ ; X_0, \alpha_k)$ of Tbet in the cell's progeny converges to zero (resp. for any $\epsilon>0$, there exists $K^*>K$ such that for all $k>K^*$, $X(k\omega; X_0, \alpha_k)\in [X_3-\epsilon,X_4+\epsilon])$.  Consequently the whole cell's progeny will develop a memory (resp. effector) profile, characterised by a low (resp. high) concentration of Tbet.  If $X_1 < X(K\omega) < X2$, cell fate depends on the values of $(\alpha_k)_{k\geq K}$.
 
 In conclusion, if the asymmetry in the repartition of proteins at division is low enough ($1>q>\max (1+\alpha^*,2-\exp (\delta \omega))$), there exist critical points in the process of differentiation toward memory or effector cell beyond which the differentiation is irreversible for the considered cell and its progeny.   Finally, at any time $t=k\omega,\ k\in \mathbb{N}$, the asymptotic state (high or low Tbet level) of a cell's lineage remains undetermined if and only if $X_1<X(t)<X_2$.



\begin{figure}[h!]
\centering
\includegraphics[scale=0.451,  angle =0 ]{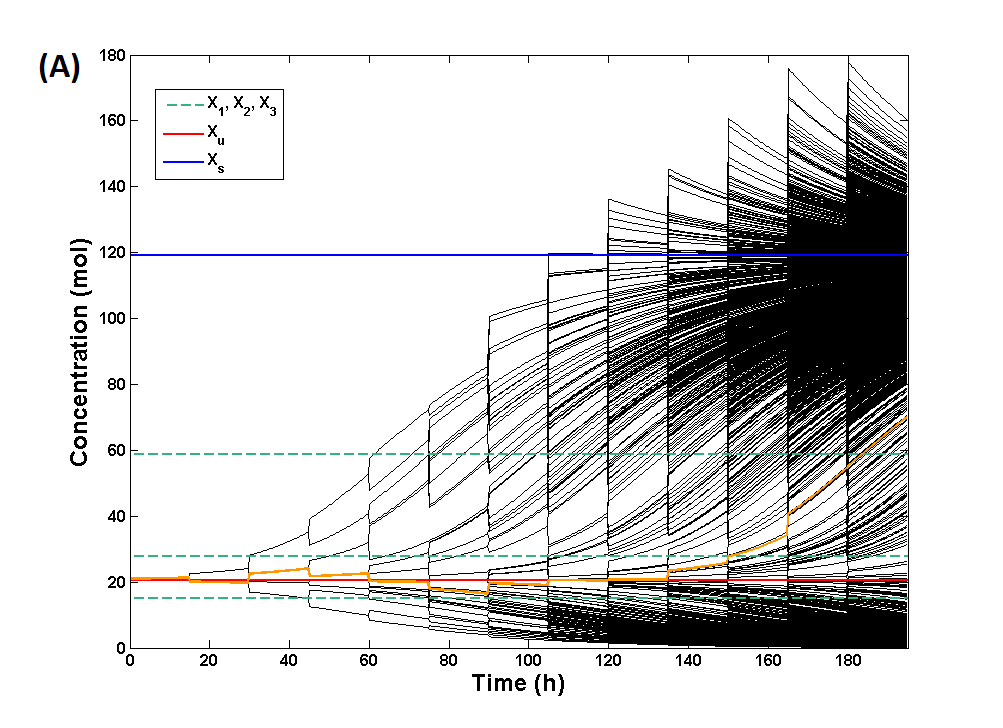}
\includegraphics[scale=0.451,  angle =0 ]{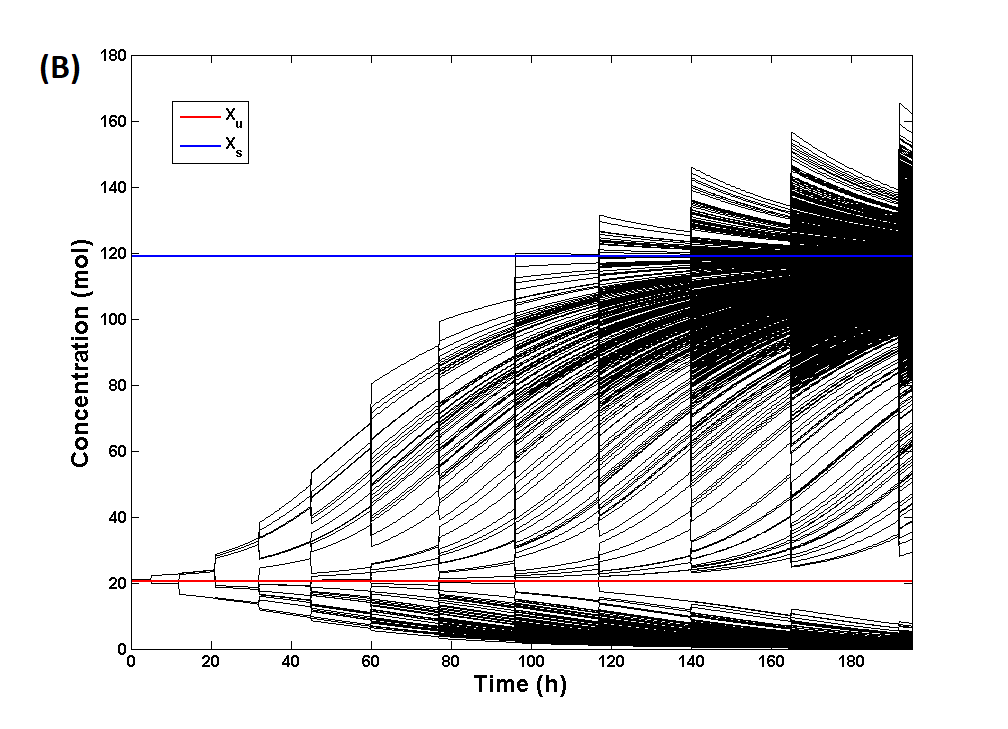}
\caption{Concentration of Tbet in a cell population generated by an initial cell with $X_0=21$ mol. (A) Constant cell cycle length $\omega = 900 \text{ min}$. (B) Increasing cell cycle length $\omega_k=300+120(k)$ min, $k\in \mathbb{N}$. Dashed green lines in (A), from bottom to top: $X_1\approx 12\text{ mol.L}^{-1}$, $X_2\approx 32\text{ mol.L}^{-1}$, $X_3\approx 56\text{ mol.L}^{-1}$ ($X_4\approx 361\text{ mol.L}^{-1}$, not shown). Orange trajectory in (A) highlights a case where Tbet concentration crosses $X_u$ several times and finally leaves the interval $[X_1,X_2]$ such that, from there, cell fate (effector) is irreversible. Red thick straight line: $X_u\approx 20\text{ mol.L}^{-1}$. Blue straight line: $X_s\approx 119\text{ mol.L}^{-1}$. }
\label{FigureSolutions}
\end{figure}

  Those results are illustrated on Figure \ref{FigureSolutions}.A. Using parameter values from (\ref{param}), $\omega=900\text{ min}$ and $q=0.8$, such that $1>q>\max (1+\alpha^*,2-\exp (\delta \omega))\approx 0.74$.
Figure \ref{FigureSolutions}.A represents the evolution of the concentration of Tbet in the whole progeny generated by an initial cell with concentration $X_0=21\in (X_1,X_2)$. At each cell division, a coefficient $\alpha_k$ is drawn from the uniform law on $(q-1,1-q)$ and sets the degree of asymmetry of the division. After a few divisions, one observes the emergence of a pool of cells with a low (lower than $X_1$) concentration for Tbet, associated to irreversible differentiation in memory cells, and a pool of cells with a high (higher than $X_2$) concentration for Tbet, associated to irreversible differentiation in effector cells. However, at time $t=12\omega=180h$, $34$ cells are still characterised by a Tbet concentration between $X_1$ and $X_2$: therefore their fates are undetermined  (see for instance the cell with the orange trajectory on Figure \ref{FigureSolutions}.A).

In practice, the cell cycle length is very short following activation but increases in the following days, when the cells have undergone some divisions \citep{Yoon2010}. It is then meaningful to consider that the cell cycle length increases after each division. 

In that case, it is easy to verify that the conclusions of Proposition~\ref{PropConclusion} remain true. In particular, for fixed values of $\alpha_{min}$ and $\alpha_{max}$, $X_1$ and $X_2$ from Proposition~\ref{PropConclusion} respectively increases and decreases when the cycle length $\omega$ increases (and converge to $X_u$ if $\omega$ tends to infinity). Consequently, while the cycle length increases,  the interval of Tbet concentration $[X_1,X_2]$ for which the cell fate remains undetermined shrinks. This suggests that increasing cell cycle length not only slows down the expansion of a cell population but also precipitates cell fate decision. 

This is illustrated on Figure \ref{FigureSolutions}.B. Parameter values are identical to those from Figure \ref{FigureSolutions}.A. except that the cell cycle length starts from $300\text{ min}$ and increases by two hours at each division, then cell cycle lengths are given by $\omega_k=300+120k\text{ min}$  for $k\in \mathbb{N}$. In particular at time $t=45h$ cells enter in a cycle of $\omega_5=900\text{ min}$ as in Figure \ref{FigureSolutions}.A. 
As cell cycle length increases, the interval of Tbet concentration in which cell fate remains undetermined shrinks (not shown) so that more and more cells adopt a definitive fate.  At time $t=160\text{ h}$, all the cell fates are irreversible.

\section*{Discussion}

In this paper, we studied the effect of unequal molecular partitioning at cell division on the emergence of phenotypic heterogeneity in a population of CD8 T-cells. To do so,
we introduced the impulsive system (\ref{ImpulsiveEquation}), characterised by a specific form of impulses but a general form of the reaction term. We proved results on the existence and stability of periodic solutions which represent attractors of the solutions and consequently are biologically relevant. Most of those results rely on the properties of the flow of an autonomous differential equation. Nevertheless, most of the time, an explicit expression of this flow cannot be determined. We then focused on the properties of the flow and gave in Theorem \ref{TheoremConvexite} sufficient conditions for the flow to be convex. Those results were applied in Sections \ref{ChapitreImpulseTbet} and \ref{SectionBio} to the case of protein Tbet regulation (described by an autonomous differential equation) and partitioning (described as impulse) in a CD8 T-cell lineage. We investigated how the degree of asymmetry in the molecular partitioning process can affect the differentiation of a CD8 T-cell toward effector or memory phenotype in Theorem \ref{TheoremeSolPeriodiques} and Proposition \ref{PropConclusion}. Associating high concentration of Tbet with effector phenotype and low concentration of Tbet with memory phenotype \citep{Joshi2007,Kaech2012,Lazarevic2013}, we showed  that if the degree of asymmetry is small enough, either the cell concentration of Tbet belongs to a non-trivial interval and the cell can still generate both effector and memory cell, or the cell differentiation is irreversible. 

This model is of course too simple to provide biologically realistic quantitative predictions, partly because the process of CD8 T-cell differentiation is too complex to be reduced to a Tbet-mediated differentiation process and partly because stochastic partitioning is not the only source of heterogeneity. 
In this regard, \citet{Huh2010} emphasised that different sources of heterogeneity (e.g. stochastic partitioning and gene expression noise) have redundant effects and therefore, evaluating the contribution of each source is a challenging task. It could then be instructive to consider an impulsive system in the form of (\ref{ImpulsiveEquationFeedback}) with a stochastic right-hand side function, accounting for gene expression noise. Regarding the simplicity of our model, it is however noticeable that it allows to give insight into a complex biological process by providing theoretical background and original answers to a paramount biological question, and consequently contributes to fill the gap between experimental biology and mathematics.

Regarding the convexity of the flow of an autonomous differential equation, we gave necessary and sufficient conditions to conclude on the strict convexity of the flow when the reaction term of the differential equation is a piecewise linear function (Proposition \ref{PropConvexitéPhin}). However, if the reaction term is continuously differentiable, we only concluded on the convexity (not necessarily strict) of the flow in Theorem \ref{TheoremConvexite}. Based on these results and numerical simulations, one can hypothesise that strict convexity can actually be obtained under Theorem \ref{TheoremConvexite}'s hypotheses. Note that, in that case, Lemma \ref{LemmeConvexitéPhi} would lead to the strict convexity (resp. concavity)  of the flow on $(0,X_c)$ (resp. $(X_c,+\infty)$) and, as an immediate consequence of Definition \ref{defXcBetaI}, we could conclude that $I$ is reduced to a single point, as observed in Figure \ref{graphGamma}. 

\begin{acknowledgements}
 This work was performed within the framework of the LABEX MILYON (ANR-10-LABX-0070) of Universit\'e  de Lyon, within the program "Investissements d'Avenir" (ANR-11-IDEX-0007) operated by the French National Research Agency (ANR). 
\end{acknowledgements}

\end{document}